\documentclass[bj,authoryear]{imsart}
\makeatletter
\renewcommand{\@biblabel}[1]{[#1]\hfill}
\usepackage{times}

\usepackage{amsbsy}
\usepackage{mathrsfs}
\usepackage{amsmath}
\usepackage{amsfonts}
\usepackage{amssymb}
\usepackage{graphicx}
\usepackage{times}
\usepackage{comment}
\usepackage{dsfont}
\usepackage{bm}

\usepackage{color}
\usepackage{fancyhdr}
\allowdisplaybreaks[4]

\usepackage{enumerate}
\usepackage{enumitem}

\newcommand{\cK}{\mathcal{K}}
\newcommand{\cO}{\mathcal{O}}
\newcommand{\cP}{\mathcal{P}}

\newcommand{\vf}{\mathbf{f}}

\newcommand{\bE}{\mathbb{E}}\newcommand{\bF}{\mathbb{F}}
\newcommand{\bH}{\mathbb{H}}

\newcommand{\bN}{\mathbb{N}}
\newcommand{\bP}{\mathbb{P}}\newcommand{\bR}{\mathbb{R}}

\newcommand{\sD}{\mathscr{D}}
\newcommand{\sE}{\mathscr{E}}

\newcommand{\sH}{\mathscr{H}}

\DeclareMathOperator{\argmax}{argmax}
\DeclareMathOperator{\argmin}{argmin}

\usepackage{caption, subcaption}
\usepackage{float}
\usepackage{graphicx}
\usepackage{xcolor}
\usepackage{hyperref}
\hypersetup{colorlinks=true,linkcolor=red,citecolor=blue,urlcolor=blue}
\usepackage{url}
\usepackage{graphicx}
\usepackage{subcaption}   
\usepackage{rotating}    
\usepackage{cleveref}
\usepackage{multirow}
\usepackage{float}
\usepackage{makecell}
\usepackage{amsmath}
\usepackage{amssymb}
\usepackage{mathrsfs}
\usepackage{placeins}
\usepackage{comment}
\usepackage{tcolorbox}
\usepackage{enumitem}
\usepackage{algorithm}
\usepackage{algpseudocode}
\usepackage{caption}
\tcbuselibrary{skins}
\newtcolorbox{roundedbox}[1][]{%
  colback=white, 
  colframe=black, 
  arc=3pt, 
  boxrule=0.8pt, 
  width=\linewidth, 
  boxsep=0.0pt, 
}

\startlocaldefs
\theoremstyle{plain}
\newtheorem{assumption}{Assumption}
\newtheorem{theorem}{Theorem}[section]
\newtheorem{proposition}[theorem]{Proposition}
\newtheorem{lemma}[theorem]{Lemma}

\newtheorem{remark}[theorem]{Remark}

\theoremstyle{definition}
\newtheorem{definition}[theorem]{Definition}

\endlocaldefs

\begin{document}

\begin{frontmatter}
\title{Smoothed estimation of Wasserstein barycenters}
\runtitle{Smoothed estimation of Wasserstein barycenters}

\begin{aug}

\author[A]{\inits{P.}\fnms{Pengtao}~\snm{Li}\ead[label=e1]{pengtaol@usc.edu}}
\author[B]{\inits{C.}\fnms{Changbo}~\snm{Zhu}\ead[label=e2]{czhu4@nd.edu}}
\author[A,C]{\inits{X}\fnms{Xiaohui}~\snm{Chen}\ead[label=e3]{xiaohuic@usc.edu}}
\address[A]{Department of Mathematics, University of Southern California, Los Angeles, CA, USA\printead[presep={,\ }]{e1,e3}}

\address[B]{Department of ACMS, University of Notre Dame, Notre Dame, IN, USA\printead[presep={,\ }]{e2}}

\address[C]{Thomas Lord Department of Computer Science, University of Southern California, Los Angeles, CA, USA}
\end{aug}

\begin{abstract}
This paper studies the statistical estimation of exact Wasserstein barycenters. Existing non-asymptotic results for empirical barycenters exhibit a severe curse of dimensionality. Motivated by the semi-dual formulation of the barycenter problem and its associated Sobolev optimization geometry, we develop a smoothness-aware approach that combines density estimation with Sobolev geometric structure to estimate the population barycenter. We establish nonparametric convergence rates for estimating both the barycenter functional and its minimizer, demonstrating how smoothness can substantially improve statistical performance.
\end{abstract}


\end{frontmatter}

\section{Introduction}

Optimal transport (OT) has emerged as a central tool for comparing and averaging probability distributions, with applications spanning statistics, machine learning, computer vision, and the natural sciences~\citep{srivastava2018scalable,ZhuangChenYang2022,rabin2012wasserstein,Solomon_2015, Gramfort_Peyre_Cuturi2015fast,zhu2024spherical,deb2024distribution}.
Among its many constructions, the Wasserstein barycenter based on the theory of OT provides a ``horizontal" notion of the mean for distribution-valued data, capturing subtle geometric structure that is insensitive to ``vertical" or pointwise averaging. Formally, the Wasserstein barycenter was introduced in~\citep{AguehCarlier2011} as an intrinsic notion of averaging under the $2$-Wasserstein metric. Given \( m \) probability measures \( \mu_1, \dots, \mu_m \) supported on compact and convex set $\Omega \subset\mathbb R^d$ with finite second moments and a barycentric coordinate vector $\boldsymbol{\omega} =(\omega_1,\dots,\omega_m)$ satisfying $\omega_j > 0$ and $\sum_{j=1}^m\omega_j=1$, the Wasserstein barycenter $\bar{\mu}_m$ is defined as any minimizer of the weighted variance functional $V : \cP_2(\Omega) \to \bR$ via
\begin{align}\label{Wasserstein_barycenter_functional}
    \min_{\mu \in \cP_2(\Omega)} \Big\{ V_{[\mu_1,\dots,\mu_m]}(\mu) := \sum_{j=1}^m \frac{\omega_j}{2}W_2^2(\mu,\mu_j) \Big\},
\end{align}
where $W_2(\mu, \nu)$ denotes the 2-Wasserstein distance between $\mu$ and $\nu$ and $\cP_2(\Omega)$ is the space of all probability measures on $\Omega$ with finite second moments. Despite its wide applications, the Wasserstein barycenter exhibits several structural properties that make both computation and statistical estimation intrinsically challenging. In contrast to the variational formulation of the Euclidean average, the objective functional $V$ is not geodesically convex in the natural $W_2$ geometry beyond two-marginal measures~\citep{ambrosio2008gradient,Chewi_COLT_Bures_Wasserstein_gradient_descent20}. Moreover in such regime $m \geq 3$, the Wasserstein barycenter has much less regularity than the Euclidean sample mean, and it is unstable to small perturbations, even when its uniqueness is ensured~\citep{KIM2017640,Santambrogio_Wang_example_16,Carlier_barycenter_24PTRF,ZhuangChenYang2022}. 

In practice, one often only has access to samples from the probability measures and must deal with the additional discretization and stochastic errors when approximating the population Wasserstein barycenter. In particular, if all input measures are discrete such as empirical distributions from samples, then the resulting barycenters are also discrete, and one can always choose a version that is provably sparse~\citep{anderes2016discrete}, imposing a fundamental obstacle in obtaining a reliable estimator of an absolutely continuous population barycenter. As a consequence of these roadblocks to performing statistical estimation of Wasserstein barycenters, most of the existing results focus on the regularized barycenter \citep{Bigot_regularized_barycenter19, CarlierEichingerKroshnin_entropic-barycenter, Chizat_doubly_regularized_barycenter25}. With entropy regularization, the multi-marginal system~\eqref{Wasserstein_barycenter_functional} is much better posed with smoothness and stability, and the estimation of the associated barycenters admits substantially stronger and more reliable statistical guarantees~\citep{Li_Chen_Multimarginal25}. Nonetheless, this comes at the cost of altering the underlying optimal transport geometry, with unfavorable dependence on the (small) regularization parameter. 

\subsection{Contributions}
This paper narrows this gap by studying the statistical sample complexity for estimating the \emph{exact} Wasserstein barycenter from point cloud data. Our motivation stems from recent advances in barycenter optimization that reveal strong duality in the Sobolev geometry for the Kantorovich potentials~\citep{kim2025sobolevgradientascentoptimal,KimYaoZhuChen2025_barycenter-nonconvex-concave}. By integrating such optimization geometry and density estimation technique, we introduce a semi-dual formulation-based approach that utilizes the smoothness of the marginal input distributions. 
Under a standard sampling model and assumptions, we establish nonparametric estimation rates for the barycenter cost functional 
$\sE(\mu_1,\dots,\mu_m) := \min_{\mu \in \cP_2(\Omega)} V(\mu)$ and the Wasserstein barycenter $\bar{\mu}_m$. While existing non-asymptotic bounds (cf.~\eqref{eqn:carlie_rate} below) for empirical Wasserstein barycenters ~\citep{Carlier_barycenter_24PTRF} apply under minimal assumptions and consequently suffer from severe curse-of-dimensionality effects, our analysis shows that additional smoothness of the marginal distributions can substantially improve statistical rates in a way similar to density estimation under the $W_p$ metric~\citep{Niles_smooth_density_minimax22}. As an extension of our result, we also derive sample complexity guarantees under a two-layer sampling model. To the best of our knowledge, this is the first work to establish smoothness-aware nonparametric rates for exact Wasserstein barycenter related quantities. 



\subsection{Related work}



There are relatively scarce works establishing non-asymptotic statistical guarantees for estimating unregularized Wasserstein barycenters from data. Suppose that for each \( j \in [m] \), one observes independent and identically distributed (i.i.d.)\ samples \( \{X_{j1}, \dots, X_{jn}\} \sim \mu_j \), and that the samples across different \( j \) are independent.
In this setting, \citet{Carlier_barycenter_24PTRF} studied the empirical Wasserstein barycenter associated with the empirical measures \( \hat{\mu}_j = n^{-1} \sum_{i=1}^n \delta_{X_{ji}}\), defined as  $\hat{\mu}^{(n)}_m \in \argmin_{\mu\in\cP_2(\bR^d)} \frac{1}{m}\sum_{j=1}^m W_2^2(\mu, \hat{\mu}_j).$ Under suitable regularity conditions, they established the following error bounds
\begin{equation}
\label{eqn:carlie_rate}
\bE W_2^2(\hat{\mu}^{(n)}_m,\bar{\mu}_m) \lesssim \begin{cases}
    n^{-1/12},   & d<4, \\
    n^{-1/12}(\log n)^{1/6},   & d=4,\\
    n^{-1/3d},   & d>4.
\end{cases}
\end{equation}
When $\mu_j$'s are further discrete measures with $N$ finite support points, an improved $\cO(\sqrt{N/n})$ rate of convergence for estimating barycenter functional was recently obtained \citep{Portales_Pauwels_Cazelles_sample_complexity_dsicrete25}.

There is another parallel line of works considering the convergence of barycenter problems with measures being directly observed (without point clouds or further discretization) and sampled in the Wasserstein space, i.e., the barycenter associated with $\bP\in\cP_2(\cP_2(\bR^d))$ defined as
\begin{equation*}
    \mu^*\in\displaystyle \argmin_{\mu\in\cP_2(\bR^d)} \int_{\cP_2(\bR^d)} W_2^2(\mu,\nu) d\bP(\nu).
\end{equation*}
Suppose that $\mu_1,\dots,\mu_m$ are random elements drawn independently from $\bP$, and every $\mu\in \text{supp}(\bP)$ is the pushforward of the barycenter $\mu^*$ by the gradient of a $\kappa$-strongly convex and $\lambda$-smooth function with $\lambda-\kappa<1$ (a.k.a. Assumption \ref{assumption_Thibaut_22_fast} in Section~\ref{Sec_two_staged_sampling}), \cite{Thibaut_Quentin_Philippe_Austin_22fast} showed that 
$\bE W_2^2(\bar{\mu}_m,\mu^*) \lesssim m^{-1},$
where $\bar{\mu}_m$ is the equally weighted barycenter of $\mu_1, \dots, \mu_m$. Specializing to Gaussian measures, \cite{Chewi_COLT_Bures_Wasserstein_gradient_descent20} derived global algorithmic convergence rates of gradient methods for estimating the Bures-Wasserstein barycenter together with their statistical sample complexity. Under milder assumptions, \cite{Carlier_barycenter_24PTRF} showed a much slower convergence rate $\bE W_2(\bar{\mu}_m,\mu^*) \lesssim m^{-1/30}.$ 



\subsection{Notations}
Let $\Omega\subset \bR^d$ be a compact and convex set. The set of probability measures on $\Omega$
with finite second moments is denoted by $\cP_2(\Omega)$.
For any $\mu\in\cP_2(\Omega)$, $T\sharp\mu$ represents the pushforward of $\mu$ by $T:\Omega\to\Omega$. With a slight abuse of notation, we also use $\mu$ to denote density for absolutely continuous probability. $f^*(x) := \sup_{y\in\Omega}\{\langle x,y \rangle - f(y)\}$ denotes the Young-Fenchel transform of $f$ extended by $+\infty$ outside of $\Omega$.
For continuous symmetric function $c:\Omega\times\Omega\to\bR$,
the $c$-transform of any $f:\Omega\to \bR$ extended by $+\infty$ outside of $\Omega$, denoted as $f^c:\Omega\to\bR$, is defined as $f^c(x):=\inf_{y\in\Omega}\{c(x,y)-f(y)\}$.
We primarily focus on the case that $c(x,y):=\|x-y\|_2^2/2$ in this paper. We use the notation $x \lesssim y$ to indicate  $x \leq C y$ for some constant  $C > 0$ that may depend on all parameters of the statistical problem except sample size $n$. $C^{k,\alpha}(\Omega)$ collects all functions with $\alpha$-H\"older $k$-th order derivatives on $\Omega$. The homogeneous Sobolev space $\dot {\bH}^1(\Omega):=\{f:\Omega\to\bR,\,\,
\int_\Omega f\,dx=0 \ \text{and}\ \|f\|_{\dot {\bH}^1(\Omega)}<\infty \}$
is a Hilbert space equipped with the $\dot {\bH}^1$-inner product $\langle f,g\rangle_{\dot {\bH}^1(\Omega)}:=\int_\Omega \langle \nabla f(x),\nabla g(x)\rangle\,dx$ and $\dot {\bH}^1(\Omega)$-norm $\|f\|_{\dot {\bH}^1(\Omega)}=\sqrt{\langle f,f\rangle_{\dot {\bH}^1(\Omega)}}$. For any two probability densities $\mu$ and $\nu$ on $\Omega$, viewed as elements of
the dual space $\dot{\bH}^{-1}(\Omega) $ of $\dot{\bH}^1(\Omega)$, we define $\|\mu-\nu\|_{\dot{\bH}^{-1}(\Omega)}
:= \sup\{
\int_\Omega \phi\, d(\mu-\nu)
\,:\,
\phi \in \dot{\bH}^1(\Omega),\ \|\phi\|_{\dot{\bH}^1(\Omega)}\le 1
\}$. When the context is clear, we omit the dependence on $\Omega$.
The $p$-Wasserstein distance $W_p$ ($p\geq1$) is defined as $W^p_p(\mu,\nu) =\inf_{\pi\in\Pi(\mu,\nu)}\int \|x-y\|^p d \pi(x,y),$
where $\Pi(\mu,\nu)$ denotes the set of all couplings of $\mu$ and $\nu$.  The inhomogeneous norm $\bH^s$, $s\in\bN$ is defined as $\|f\|_{\bH^s} := ( \sum_{|\gamma|\leq s} \| D^{(\gamma)} f\|_{2}^2 )^{1/2},$
where $D^{(\gamma)} f :=\frac{\partial^{|\gamma|}}{\partial x_1^{\gamma_1} \dots\partial x_d^{\gamma_d} } f$ for $|\gamma| = \sum_i \gamma_i$ for a multi-index $\gamma$. For any random vector $\xi$, we denote its law by $\mathcal{L}(\xi)$. For positive integer
$m$, we write $[m] := \{1,\dots, m\}$. Throughout the paper, we assume $\Omega=[0,1]^d$ without loss of generality.

\subsection{Organization}
The rest of the paper is organized as follows. 
The smoothness-aware semi-dual approach for barycenter estimation is introduced in Section \ref{Section_smoothness_adaptive_semi_dual_based_approach}. The main theorems stating the estimation error bounds of barycenter-related quantities are presented in Section \ref{Sec_sample_complexity}. We apply our results to a two-layer sampling model in Section \ref{Sec_two_staged_sampling}. Section \ref{sec:numerical_experiment} reports numerical experiments demonstrating the advantage of our estimator. The Appendix contains proofs that are omitted from the main text and associated technical tools.

\section{Smoothness-aware semi-dual approach}\label{Section_smoothness_adaptive_semi_dual_based_approach}
In this section, we introduce the semi-dual formulation-based approach for barycenter estimation that leverages the smoothness of the marginal input distributions. The resulting estimator is shown to mitigate the curse-of-dimensionality in Section \ref{Sec_sample_complexity}.

\subsection{Semi-dual formulation of barycenter cost functional}
We begin with the dual formulation of the Wasserstein barycenter functional $\sE(\mu_1,\dots,\mu_m) = \min_{\mu \in \cP_2(\Omega)} V_{[\mu_1,\dots,\mu_m]}(\mu)$ in (\ref{Wasserstein_barycenter_functional}). Define
\begin{equation}\label{Wasserstein_barycenter_dual_functional}
\sD_{[\mu_1,\dots,\mu_m]}(f_1, \ldots, f_{m-1}) = \sum_{i=1}^{m-1} \omega_i \int f_i^c \, d\mu_i + \omega_m \int f_{\mathrm{mix}}^c \, d\mu_m,
\end{equation}
where
$f_{\mathrm{mix}} = - \sum_{j=1}^{m-1} \frac{\omega_j}{\omega_m} f_j$. For notational convenience, we sometimes denote $f_m:= f_{\mathrm{mix}}$ and evaluation of the objective functional $\sD_{[\mu_1,\dots,\mu_m]}(\cdot)$ at $\vf = (f_1, \ldots, f_{m-1}, f_{\mathrm{mix}})$ means evaluation at its first $m-1$ components.
As shown by \cite{Bohan_Matt_signed_barycenter26,kim2025sobolevgradientascentoptimal}, for absolutely continuous probabilities $\mu_1,\dots,\mu_m$, strong duality holds:
\begin{align}\label{equation_strong_duality}
\sE (\mu_1,\dots,\mu_m) = \sup_{f_1,\dots, f_{m-1}} \sD_{[\mu_1,\dots,\mu_m]}(f_1,\dots,f_{m-1}),
\end{align}
where the supremum is over continuous functions $f_{1},\dots,f_{m-1}$. One could easily check that for $a_j\in\bR$, $(f_{*,1}+a_1,\dots,f_{*,m-1}+a_{m-1})$ also optimizes the equation above if 
$(f_{*,1},\dots,f_{*,m-1})$ does. To account for this translation invariance of dual potentials, we distinguish the optimal dual potential $(f_{*,1},\dots,f_{*,m-1})$ with $\sup_{x_i\in\Omega}f_{*,i}(x_i) = 0,\, i\in[m-1]$ from now on. We refer to $\vf_* := (f_{*,1},\dots, f_{*,m})$ as the Kantorovich potential and the unique Wasserstein barycenter $\bar{\mu}_m$ can be characterized by 
\begin{align}\label{primal_dual_relation}
    \bar{\mu}_m = T_{f^c_{*,i}}\sharp\mu_i, 
    \quad
    T_{f_{*,i}} \sharp\bar\mu_m=\mu_i,
\end{align}
where $T_h := \text{id} - \nabla h$. 
We denote the support of $\bar\mu_m$ as $\operatorname{supp}(\bar\mu_m):= \cK \subset \Omega$.

\subsection{Smoothed barycenter estimation}


Based on the semi-dual formulation (\ref{equation_strong_duality}) of the barycenter functional, we are now ready to present our smoothing approach for barycenter estimation in this section. We work with samples $\{X_{ji} : j \in [m],\, i \in [n]\}$ 
and first construct a smoothed estimator of the marginal probability density from the data. Concretely, it is obtained by expanding the unknown density in a sufficiently regular wavelet basis \citep{Tsybakov_wavelets_textbook}, estimating each wavelet coefficient via its empirical average from the samples, and then truncating the resulting expansion \citep{Niles_smooth_density_minimax22,Tengyuan_Liang_JMLR21,NEURIPS2018_Singh_Uppal_Li_Li_Zaheer_Poczos,NEURIPS2019_Uppal_Singh_Poczos}. In particular, for density $\mu_j$ with upper and lower bounds and certain regularity (i.e., Assumptions \ref{assumption_upper_lower_bounded_density}--\ref{assumption_smooth_besov_density} below), one can construct a smoothed version $\tilde{\mu}_j$ of the empirical distribution with required level of accuracy as in Lemma \ref{Lemma:jonathan_smooth_dentisy_estimation}.
Then we solve the semi-dual \eqref{Wasserstein_barycenter_dual_functional} with the smoothing estimate of the marginal densities and reconstruct the barycenter using the pushforward relation~\eqref{primal_dual_relation}. Define 
\begin{multline}
 \bF_{\alpha} := \Bigl\{  f \in C^{1,1}(\Omega) \; \big| \; \text{ for any}\: x,y \in\Omega, \; \varphi(\cdot) := \frac{\|\cdot\|^2}{2} - f (\cdot),\\
\frac{\alpha}{2}\,\|x - y \|_2^2\leq \varphi(x)-\varphi(y)-\langle \nabla \varphi(y), x - y\rangle
\Bigr\} .\nonumber
\end{multline}
Our pseudo-algorithm is summarized in Algorithm~\ref{alg:main}.

\begin{algorithm}
\caption{Smoothness-aware barycenter estimation procedure.\label{alg:main}}
\begin{algorithmic}[1]
\Require Data points $X_{j1},\dots, X_{jn}$ for $j \in [m]$.

\State Construct smoothed density estimator $\tilde{\mu}_j$ from $X_{j1},\dots, X_{jn}$ using truncated wavelet expansion.

\State Maximize $\sD_{[\tilde\mu_1,\dots,\tilde\mu_m]}(\vf) := \sum_{i=1}^{m-1} \omega_i \int f_i^c \, d\tilde{\mu}_i + \omega_m \int f_{\mathrm{mix}}^c \, d\tilde{\mu}_m$ 
and extract the solution $ ( \tilde{f}_{*,1}, \dots,\tilde{f}_{*,\mathrm{mix}}) $.

\State Set $\tilde{\vf}_* =( \tilde{f}_{*,1}, \dots,\tilde{f}_{*,\mathrm{mix}} )$ and $\widetilde{\sE}_n=\sD_{[\tilde\mu_1,\dots,\tilde\mu_m]}(\tilde{\vf}_*)$.

\State\Return Smoothed empirical barycenter $\tilde{\mu}_m^{(n)} = T_{\tilde{f}_{*,i}^c} \sharp \tilde{\mu}_i$.
\end{algorithmic}
\end{algorithm}

In practice, the optimization step in Algorithm~\ref{alg:main} is solved using a Sobolev gradient ascent (SGA) algorithm proposed in~\cite{kim2025sobolevgradientascentoptimal}, which established a sub-linear algorithmic convergence rate for absolutely continuous marginal input distributions.
We remark that our procedure is agnostic to the construction of an estimator 
$\tilde{\mu}_j$
from discrete samples, namely one can adopt other density estimators than the wavelet expansion, and our analysis is robust to the choice of the initial density estimation.

\subsection{Sobolev optimization geometry}\label{Sec_Strong_concavity_optimization_geometry}
In this section, we recognize an appropriate optimization dual
geometry for the Wasserstein barycenter problem (\ref{Wasserstein_barycenter_functional}), which will serve as a foundation for subsequent statistical complexity in Section \ref{Sec_sample_complexity}. Specifically, by identifying the Sobolev geometric structure of the semi-dual functional $\sD_{[\mu_1,\dots,\mu_m]}(\cdot)$ that
respects the structure of the Wasserstein barycenter, we pave the way for nonparametric rates of convergence.

We endow the dual optimization problem \eqref{Wasserstein_barycenter_dual_functional}
with a Sobolev product–Hilbert geometry, which will play a central role in our statistical analysis.
Let $\mu_1,\dots,\mu_m$ and
$\tilde\mu_1,\dots,\tilde\mu_m$ be absolutely
continuous probability measures supported on
$\Omega$, with densities satisfying
$L\le \mu_i,\tilde\mu_i\le U$ for all $i\in[m]$.
Assume that $(\mu_1,\dots,\mu_m)$ admits
the Kantorovich potential
$\vf_*=(f_{*,1},\dots,f_{*,m-1},f_{*,\mathrm{mix}})\in\bF_\alpha^m$.
Recall that $\cK$ denotes the support of barycenter $\bar\mu_m$ associated with $(\mu_1,\dots,\mu_m)$.
Specifically, let
$
\sH_{m-1} := (\dot{\bH}^1(\cK))^{m-1}
$
denote the Cartesian product of $(m-1)$ copies of the homogeneous Sobolev space
$\dot{\bH}^1(\cK)$. 
For
$\bm{\varphi} := (\varphi_1,\ldots,\varphi_{m-1}), \bm{\psi} := (\psi_1,\ldots,\psi_{m-1}) \in \sH_{m-1}$, we define the
weighted inner product
$
\langle \bm{\varphi}, \bm{\psi} \rangle_{\sH_{m-1}}
:= \sum_{i=1}^{m-1} \omega_i
\int_{\cK} \langle \nabla \varphi_i(x), \nabla \psi_i(x) \rangle \, dx,
$
and the associated norm $
\| \bm{\varphi} \|_{\sH_{m-1}}^2
:= \sum_{i=1}^{m-1} \omega_i
\int_{\cK} \| \nabla \varphi_i(x) \|^2 \, dx$. Since $\sH_{m-1}$ is a finite product of Hilbert spaces,
its dual space is identified as $\sH_{m-1}' = (\dot \bH^{-1}(\cK))^{m-1}$, where the duality pairing is taken componentwise. We shall consider the gradient of $\sD_{[\tilde\mu_1,\dots,\tilde\mu_m]} (\vf_*)$ as an element in the dual space $\sH_{m-1}'$.
\begin{definition}[Gradient of the dual objective] \label{Gradient_of_the_dual_objective}
Define $\nabla \sD_{[\tilde\mu_1,\dots,\tilde\mu_m]} (\vf_*) \in \sH_{m-1}'$ as
\begin{align}
\left \langle \nabla \sD_{[\tilde\mu_1,\dots,\tilde\mu_m]} (\vf_*), (\varphi_1, \ldots, \varphi_{m-1}) \right\rangle
:= \sum_{i=1}^{m-1} \omega_i \int \varphi_i \, 
d\, \left( T_{f_{*,\mathrm{mix}}^c}\sharp \tilde\mu_m - T_{f_{*,i}^c}\sharp\tilde\mu_i \right).
\end{align}
\end{definition}

\noindent In particular, we readily get the following expression for the norm of the above gradient.
\begin{proposition}[Norm of dual objective gradient]\label{proposition_Norm_dual_objective_gradient}
 $$\|\nabla \sD_{[\tilde\mu_1,\dots,\tilde\mu_m]} (\vf_*)\|_{\sH_{m-1}'} 
 = 
 \left( \sum_{i=1}^{m-1} \omega_i \|T_{f^c_{*,i}}\sharp\tilde\mu_i-T_{f^c_{*,\mathrm{mix}}} \sharp\tilde\mu_m\|^2_{\dot{\bH}^{-1}(\cK)} \right)^{1/2}.$$ 
\end{proposition}
Proposition~\ref{proposition_Norm_dual_objective_gradient} expresses the gradient norm as a weighted negative Sobolev discrepancy among the perturbed marginals pushed forward by the population transport maps. For the population marginals, these pushforwards all coincide with the barycenter $\bar\mu_m$ by \eqref{primal_dual_relation}. Their discrepancy under perturbation therefore quantifies how far the population potentials are from satisfying the optimality conditions of the perturbed problem. We next relate this term to the corresponding dual suboptimality gap.

Such connections between transport discrepancies and objective gaps are central to quantitative stability estimates in optimal transport. For instance, \citet[Lemma~5.1]{ghosal_sen2022AOSmultivariate} bound the squared $L^2$ distance between two transport maps by the difference of transport costs plus a linear correction term, under a strong convexity assumption. Related estimates appear in \citet{Hutter_Rigollet_minimax_oT_map21,MuzellecVacherBachVialardRudi2021,ManoleBalakrishnanNilesWeedWasserman2024,deb_ghpsal_sen2021rate}; see also \citet{Letrouit2025} for a survey and \citet{Carlier_barycenter_24PTRF} for stability results for Wasserstein barycenters.

Building on these stability ideas, the following lemma establishes a Polyak-\L{}ojasiewicz-type inequality around the population Kantorovich potential under the Sobolev geometry: the suboptimality gap of the perturbed dual objective is bounded by the squared norm of its gradient. The proof combines the classical Fenchel-gap estimates from the OT map literature mentioned above with a coupling construction adapted to the barycenter problem. This inequality supplies the perturbation control needed for our subsequent statistical analysis.

\begin{lemma}[Polyak-\L{}ojasiewicz type inequality for the dual objective]\label{Lemma:Strong_concavity}

Given absolutely continuous probability distributions $\mu_1,\dots,\mu_m$ and $\tilde\mu_1,\dots,\tilde\mu_m$ with  $L \leq \mu_i,\tilde\mu_i\leq U, i\in[m]$ and their associated Kantorovich potentials  $\vf_* := (f_{*,1},\dots, f_{*,m-1}, f_{*,\mathrm{mix}})\in \bF_{\alpha}^m$ and $\tilde\vf_* := (\tilde f_{*,1},\dots, \tilde f_{*,m-1},\tilde f_{*,\mathrm{mix}})$, it holds that 
\begin{equation}\label{strong_concavity}
\sD_{[\tilde\mu_1,\dots,\tilde\mu_m]}(\tilde\vf_*) - \sD_{[\tilde\mu_1,\dots,\tilde\mu_m]}(\vf_*) 
\lesssim 
\|\nabla \sD_{[\tilde\mu_1,\dots,\tilde\mu_m]} (\vf_*)\|^2_{\sH_{m-1}'} .
\end{equation}
Here, $\lesssim$ hides constants depending on $d,\alpha,L,U$.
\end{lemma}

\section{Sample complexity analysis}\label{Sec_sample_complexity}
In this section, we study the convergence rate of the proposed estimator $\tilde\mu_m^{(n)}$ to the population Wasserstein barycenter $\bar\mu_m$ as a function of the per-distribution sample size $n$. Under standard assumptions in the literature, we establish nonparametric estimation rates for both the barycenter functional and the Wasserstein barycenter.

\subsection{Assumptions}
We begin by introducing a set of regularity assumptions on the input measures $\{\mu_i\}_{i=1}^m$. These conditions ensure well-posedness of the barycenter problem, stability of optimal transport maps, and control of the statistical error induced by finite sampling.

\begin{assumption}\label{assumption_upper_lower_bounded_density} 
For each $i\in [m]$, $\mu_i$ is compactly supported on $\Omega$ and absolutely continuous with its density $L\leq \mu_i \leq U,$ for some $L,U\in(0,\infty)$. 
\end{assumption}

\begin{assumption}\label{assumption_smooth_besov_density}
  $\mu_i\in \bH^{s}(\ell;L)$ for some constants $s,\ell,L\in (0,+\infty)$, with 
        \begin{align*}
    \bH^{s}(\ell;L) &:= \bH^{s}(\ell) \cap \{ f, f\geq L\},  \\
    \bH^{s}(\ell) &:= \{ f\in L^2(\Omega), \|f\|_{\bH^{s}}\leq \ell, \int f=1 , f\geq 0 \}.
\end{align*}
\end{assumption}

Assumptions \ref{assumption_upper_lower_bounded_density}--\ref{assumption_smooth_besov_density} have been widely used in the nonparametric density estimation literature \citep{Niles_smooth_density_minimax22, Tsybakov_nonparametric_2009}, as well as in the optimal transport literature \citep{Hutter_Rigollet_minimax_oT_map21, Delalande_semi_OT22, Figali_OT_map15, Chewi_statistical_OT_book, Carlier_barycenter_24PTRF}.


\begin{assumption}\label{assumption_Wasserstein_potential}
The Kantorovich potential $\vf_* \in \bF_{\alpha}^m$ for some $0<\alpha<\infty$.  
\end{assumption}



Assumption~\ref{assumption_Wasserstein_potential} is a regularity condition on the Brenier
potentials transporting the Wasserstein barycenter $\bar\mu_m$ to each input measure $\mu_i$.
Similar regularity assumptions have been exploited in the study of optimal transport,
both for convergence analysis and for the design of computational algorithms. For example, a
similar condition (see Assumption \ref{assumption_Thibaut_22_fast}) was adopted by \citet{Thibaut_Quentin_Philippe_Austin_22fast} to
derive parametric rates of convergence for empirical barycenters as the number
of input measures $m$ increases. Motivated by the regularity theory of Caffarelli
\citep{caffarelli1996boundary}, \citet{paty2020regularity} proposed an algorithm for optimal
transport between discrete measures by enforcing the Brenier potential $\varphi_i$ to be
strongly convex and smooth. Under the same type of regularity assumptions, \citet{Hutter_Rigollet_minimax_oT_map21}
derived minimax estimation rates for the Brenier potential.

We highlight that the present work studies the sample complexity of Wasserstein barycenters under a milder regularity assumption. Whereas the aforementioned results explicitly require both the strong convexity and smoothness of the relevant Brenier potentials, our convergence bounds under Assumption~\ref{assumption_Wasserstein_potential} depend on curvature only through the strong convexity parameter $\alpha$, without requiring a separate quantitative smoothness bound.

We next describe sufficient conditions on the input measures and their barycenter that ensure this assumption holds. Specifically, let $\Omega=[0,1]^d$ and suppose that $\mu_i$ and
$\bar\mu_m$ have densities satisfying
\[
0<L\le \mu_i(x),\bar\mu_m(x)\le U<\infty,
\qquad x\in\Omega,
\]
and $\mu_i,\bar\mu_m\in C^{1,\tau}(\Omega)$ for some
$\tau\in(0,1)$ and all $i=1,\ldots,m$.
Let $\varphi_i$ be the Brenier potential such that
$(\nabla\varphi_i)\sharp\bar\mu_m=\mu_i$.
By the regularity theorem for optimal transport between
hypercubes \citep[Theorem~1.1]{chen2023regularity},
$\varphi_i\in C^{3,\tau}(\Omega)$.
Moreover, the Monge--Amp\`ere equation yields
\[
\det D^2\varphi_i(x)
=
\frac{\bar\mu_m(x)}
{\mu_i(\nabla\varphi_i(x))}
\ge \frac{L}{U}>0.
\]
Since $\varphi_i$ is convex, its Hessian is positive semidefinite,
and the positive determinant implies that it is positive definite
at every point of $\Omega$.
By continuity of the Hessians and compactness of $\Omega$,
$\alpha:=\min_{1\le i\le m}\min_{x\in\Omega}
\lambda_{\min}(D^2\varphi_i(x))>0$.
Thus, each $\varphi_i$ is $\alpha$-strongly convex on $\Omega$,
as required by Assumption~3. The density conditions imposed on the input measures in this sufficient criterion are satisfied, for example, by Gaussian densities restricted and normalized to $\Omega$, or by suitably chosen kernel density estimators restricted and normalized to $\Omega$. Applying the criterion also requires the barycenter $\bar\mu_m$ to have a density bounded above and below on $\Omega$ and belonging to $C^{1,\tau}(\Omega)$. These properties of the barycenter are additional hypotheses.

To generate a family of distributions $\{\mu_1,\ldots,\mu_m\}$
satisfying Assumptions~\ref{assumption_upper_lower_bounded_density}--\ref{assumption_Wasserstein_potential}, let $\mu_0$ be a reference probability
measure with $\mu_0\in C^\infty({\Omega})$, where
$\Omega=[0,1]^d$, satisfying $0<L_0\le\mu_0\le U_0<\infty$.
Given positive weights $\omega_1,\ldots,\omega_m$ summing to one,
choose functions $h_i\in C_c^\infty((0,1)^d)$ satisfying
$\sum_{i=1}^m\omega_i h_i=0$.
This can be achieved by choosing $h_1,\ldots,h_{m-1}$ freely and
setting $h_m=-\frac{1}{\omega_m}\sum_{i=1}^{m-1}\omega_i h_i$.
Extend each $h_i$ by zero to $\mathbb R^d$, and define
\[
\varphi_i(x)=\frac12\|x\|_2^2+\varepsilon h_i(x),
\qquad
T_i=\nabla\varphi_i,
\qquad
\mu_i=(T_i)\sharp\mu_0.
\]
Let $M=\max_{1\le i\le m}\sup_{x\in\mathbb R^d}
\|D^2h_i(x)\|_{\mathrm{op}}$.
Choosing $\varepsilon>0$ such that $\varepsilon M<1$ gives
$
(1-\varepsilon M)I
\preceq D^2\varphi_i(x)
\preceq(1+\varepsilon M)I$.
Thus, each $\varphi_i$ is $\alpha$-strongly convex with
$\alpha=1-\varepsilon M>0$, and $T_i$ is the optimal transport
map from $\mu_0$ to $\mu_i$. Each $T_i$ is a smooth diffeomorphism of $\mathbb R^d$ and equals
the identity outside a compact subset of $(0,1)^d$.
Consequently, it restricts to a smooth diffeomorphism of $\Omega$
onto itself. The change-of-variables formula gives
$\mu_i(T_i(x))=\mu_0(x)/\det DT_i(x)$. Consequently,
\[
\frac{L_0}{(1+\varepsilon M)^d}
\le \mu_i
\le
\frac{U_0}{(1-\varepsilon M)^d}.
\]
Moreover, $\mu_i\in C^\infty(\Omega)$, so
Assumptions~\ref{assumption_upper_lower_bounded_density}--\ref{assumption_smooth_besov_density} hold for suitable common constants and any
fixed finite smoothness order. The centering condition implies
$\sum_{i=1}^m\omega_i\varphi_i(x)=\frac12\|x\|_2^2$
for all $x\in\mathbb R^d$.
By the barycenter characterization in
\cite[Proposition~3.8]{AguehCarlier2011}, the Wasserstein
barycenter of $\mu_1,\ldots,\mu_m$ with weights
$\omega_1,\ldots,\omega_m$ is $\bar\mu_m=\mu_0$.
Define the associated Kantorovich potentials by
$f_i(x)=\frac12\|x\|_2^2-\varphi_i(x)=-\varepsilon h_i(x)$.
Then $T_{f_i}=\mathrm{id}-\nabla f_i=T_i$ and
$\sum_{i=1}^m\omega_i f_i=0$, so that
$f_m=-\sum_{i=1}^{m-1}\frac{\omega_i}{\omega_m}f_i
=f_{\mathrm{mix}}$.
After the additive normalization introduced in Section~2.1,
these potentials yield $\mathbf f_*\in\bF_\alpha^m$,
establishing Assumption~\ref{assumption_Wasserstein_potential}.

\subsection{Main results}
This section presents our main results on the statistical sample complexity of estimating the Wasserstein barycenter $\bar{\mu}_m$ and related quantities under the following setting. The measures $\mu_1,\dots,\mu_m$ are treated as fixed, and for each $j\in[m]$ we observe an i.i.d.\ sample $X_{j1},\dots,X_{jn}$ drawn from $\mu_j$. All observations are independent across both $j$ and sample indices. 

We begin by establishing several auxiliary results concerning the smoothed empirical distributions and the associated Kantorovich potentials $\vf_*$. The first result concerns the existence of a density estimator with optimal convergence rates under both the Wasserstein distance and the homogeneous negative Sobolev norm. This lemma formalizes the fact that over smooth density classes with uniform upper and lower bounds, minimax-optimal (up to logarithmic factors) estimation in $W_2$ automatically yields optimal rates in the $\dot{\bH}^{-1}$ norm. The equivalence between these two metrics under boundedness assumptions allows us to treat them interchangeably up to constants, which will be convenient in later arguments.

\begin{lemma}[Smoothed empirical marginal distribution]\label{Lemma:jonathan_smooth_dentisy_estimation}
For any $s \ge 0$, there exists an estimator $\tilde{\mu} \in L^2(\Omega)$ of the density $\mu$, constructed from an i.i.d.\ sample $X_1,\dots,X_n \sim \mu$, such that $L \le \tilde{\mu} \le U$ almost everywhere and, for any $\ell>0$,

\[
\sup_{\mu \in \bH^{s}(\ell;L,U)}
\bE \|\mu - \tilde \mu\|^2_{\dot{\bH}^{-1}} \asymp
\sup_{\mu \in \bH^{s}(\ell;L,U)}
\bE W^2_2(\mu,\tilde \mu)
\;\lesssim\;
\begin{cases}
n^{-\frac{2+2s}{d+2s}}, & d \ge 3, \\[0.6em]
n^{-1}(\log n)^2,       & d = 2, \\[0.4em]
n^{-1},             & d = 1 .
\end{cases}
\] 
Here, 
$\bH^{s}(\ell;L,U) := \bH^{s}(\ell) \cap \{ f, L\leq f \leq U\}$
and $\lesssim$ hides dependence on $\ell,s,d,L,U$.
\end{lemma}

The next results establish boundedness and curvature properties of Kantorovich potentials associated with the Wasserstein barycenter. 
\begin{proposition}[Bounded Kantorovich potentials]\label{prop_bounded_dual_potentials}
Let $(f_{*,1},\dots,f_{*,m})$ be the  Kantorovich potential associated with input marginal probability distributions $\mu_1,\dots,\mu_m$.
There exist constants $C_1$ and $C_2$ depending on the diameter of $\Omega$ and barycentric coordinate $\boldsymbol{\omega}$, such that for $i\in [m]$,
\begin{align*}
           \|f_{*,i}\|_\infty \leq C_1 \quad \text{and} \quad \|\nabla f^c_{*,i}\|_\infty  \leq C_2.
\end{align*}
\end{proposition}

\begin{lemma}[Curvature of Kantorovich potentials]
\label{lemma:population_maps_change_of_variables}
Suppose that 
$0<L\le \mu_j\le U<\infty$ almost everywhere on $\Omega$ and its associated Kantorovich potentials  $\vf_* := (f_{*,1},\dots, f_{*,m})\in \bF_{\alpha}^m$.
The barycenter $\bar\mu_m$ has a density, satisfying
\begin{equation}\label{eq:population_jacobian_formula}
 L\alpha^d\le\bar{\mu}_m(x)\le U
 \quad\text{for a.e. }x\in \cK.
\end{equation}
Moreover,
$T_{f_{*,j}}|_{\mathcal K}:\cK\to\Omega$ and $T_{f^c_{*,j}}:\Omega\to \cK$ are inverse
bi-Lipschitz maps, particularly,
\begin{equation}\label{eq:population_derivative_bounds}
\begin{gathered}
 \alpha I\preceq DT_{f_{*,j}}\preceq \frac{U}{L\alpha^{d-1}}I\quad\text{a.e. on }\,\cK,\quad
 \frac{L\alpha^{d-1}}{U}I\preceq DT_{f^c_{*,j}}\preceq\alpha^{-1}I
 \quad\text{a.e. on }\,\Omega.
\end{gathered}
\end{equation}
\end{lemma}

Building on these auxiliary results, we establish the nonparametric convergence rate of $\widetilde{\sE}_n$ to $\sE(\mu_1,\dots,\mu_m)$. The formal statement is given below, while the detailed proof is deferred until after Remark~\ref{remark:optimization_error}.

\begin{theorem}[Nonparametric rate of barycenter functional estimation]\label{thm_sample_complexity} Under Assumptions \ref{assumption_upper_lower_bounded_density}--\ref{assumption_smooth_besov_density}, the estimator $\widetilde{\sE}_n$ of the barycenter functional $\sE (\mu_1,\dots,\mu_m)$ satisfies 
\begin{align*}
    \bE\left[ \widetilde{\sE}_n - \sE (\mu_1,\dots,\mu_m) \right]^2 \lesssim   r_n,
\end{align*}
where 
\begin{align*}
r_n = 
 \begin{cases}
n^{-\frac{2+2s}{d+2s}}, & d \ge 3, \\[0.6em]
n^{-1}(\log n)^2,       & d = 2, \\[0.4em]
n^{-1},             & d = 1 .
\end{cases}
\end{align*}
Here, $\lesssim$ hides constants depending on $\ell,s,d,L,U$.
\end{theorem}

The rate $r_n$ captures the interplay between smoothness and dimensionality. For $d \ge 3$, the estimator $\widetilde{\sE}_n$ achieves the nonparametric rate $n^{-(2+2s)/(d+2s)}$, which improves as the smoothness parameter $s$ increases. As $s \to \infty$, the exponent $\frac{2+2s}{d+2s}$ converges to $1$, so that the convergence rate approaches the parametric rate $n^{-1}$, effectively eliminating the curse of dimensionality. Our main result establishes nonparametric convergence guarantees for estimating the Wasserstein barycenter $\bar{\mu}_m$, as stated below.

\begin{theorem}[Smoothness-aware barycenter estimation]\label{thm_W_1_discrepancy}
Under Assumptions \ref{assumption_upper_lower_bounded_density}--\ref{assumption_Wasserstein_potential}, the smoothed empirical barycenter $\tilde{\mu}_m^{(n)}$ and Wasserstein barycenter $\bar{\mu}_m$ admit the bound
\begin{align*}
    \bE W_2^2\Bigl( \tilde{\mu}_m^{(n)},\bar{\mu}_m \Bigr) \lesssim  
    r_n.
\end{align*}
Here, $\lesssim$ hides constants depending on $\alpha,\ell,s,d,L,U$.
\end{theorem}

\begin{remark}[Smoothness-aware barycenter estimation
under relaxed assumptions]
\label{thm_W_1_discrepancy_partial}
The strong convexity requirement in Assumption \ref{assumption_Wasserstein_potential} can be weakened to a nonempty subset of 
associated population Brenier potentials, at the cost of a slower convergence rate of $r_n^{1/2}$ instead of $r_n$.
Specifically, under Assumptions \ref{assumption_upper_lower_bounded_density}-- \ref{assumption_smooth_besov_density},
suppose that there exist a nonempty subset $J\subseteq[m]$ and a constant $\alpha>0$ such that $f_{*,j}\in \bF_\alpha$ for every $j\in J$. 
Then the smoothed empirical barycenter $\tilde\mu_m^{(n)}$ and the population Wasserstein barycenter $\bar\mu_m$ satisfy
\[
  \bE W_2^2\bigl(\tilde\mu_m^{(n)},\bar\mu_m\bigr)
  \lesssim r_n^{1/2}.
\]
The implicit constants depend on $\alpha,\ell,J,s,d,\Omega,L,U$. This extends the convergence guarantee to settings where some population dual potentials may fail to be strongly convex, at the cost of a slower rate.   
The proof is analogous to that of Theorem \ref{thm_W_1_discrepancy} and is provided in Appendix \ref{appendix:deferred_proofs}.
\end{remark}

\begin{remark}\label{rmk:boundvariations}
To the best of our knowledge, Theorem~\ref{thm_W_1_discrepancy} is the first result to explicitly incorporate the marginal smoothness $s$ into the barycenter estimation rate, thereby quantifying how smoothness can improve sample complexity.
Under Assumptions~\ref{assumption_upper_lower_bounded_density}--\ref{assumption_Wasserstein_potential},
our smoothed empirical barycenter significantly alleviates the curse of dimensionality compared with the empirical barycenter studied in \cite{Carlier_barycenter_24PTRF}, whose sample complexity bound \eqref{eqn:carlie_rate} implies
exponentially poor dimension dependence.

Smoothness also yields improved guarantees under milder conditions.
Specifically, under Assumption~\ref{assumption_upper_lower_bounded_density}--\ref{assumption_smooth_besov_density}
and Assumption~1.3 of \cite{Carlier_barycenter_24PTRF},
an application of Theorem~1.5 therein yields
\begin{align*}
    \bE W_2^2\Bigl(\tilde{\mu}_m^{(n)},\bar{\mu}_m\Bigr)
    \lesssim
    \begin{cases}
        n^{-\frac{1+s}{3(d+2s)}}, & d \geq 3, \\[0.6em]
        n^{-1/6}(\log n)^{1/3}, & d = 2, \\[0.4em]
        n^{-1/6}, & d = 1.
    \end{cases}
\end{align*}
Comparing this bound with \eqref{eqn:carlie_rate}, established in
\cite[Corollary~1.6]{Carlier_barycenter_24PTRF}, further illustrates
how marginal smoothness can improve barycenter estimation rates
and mitigate the curse of dimensionality.
\end{remark}

\begin{remark}[Effect of numerical optimization errors on the statistical guarantee]
\label{remark:optimization_error}
Under Assumptions~\ref{assumption_upper_lower_bounded_density}--\ref{assumption_Wasserstein_potential}, consider the SGA method of
\cite{kim2025sobolevgradientascentoptimal} applied to
$\widetilde \sD (\cdot):=\sD_{[\tilde\mu_1,\ldots,\tilde\mu_m]}(\cdot)$. Let $\mathcal H_\Omega=(\dot \bH^1(\Omega))^{m-1}$, equipped with
$\|\vf\|_{\mathcal H_\Omega}^2
:=\sum_{i=1}^{m-1}\int_\Omega\|\nabla f_i(x)\|_2^2\,dx$,
as in \cite{kim2025sobolevgradientascentoptimal}.
Consider the exact updates
\[
\vf^{(t+1)}
=
\vf^{(t)}+\eta_t\mathbf g_t,
\qquad
\mathbf g_t
=
\nabla_{\mathcal H_\Omega}\widetilde \sD(\vf^{(t)}),
\qquad t=0,\ldots,T-1,
\]
with deterministic step sizes $\eta_t>0$. For $i\in[m-1]$, the $i$-th coordinate of $\nabla_{\mathcal H_\Omega}\widetilde\sD(\vf)$ is the $\dot\bH^{1}$ gradient of $\widetilde\sD$ at $f_i$, defined by 
$ (\nabla_{\mathcal H_\Omega}\widetilde\sD (\vf) )_i
= (-\Delta)^{-1} \bigl( -\omega_i\bigl( (T_{f_i^c})\sharp\tilde\mu_i
    - (T_{f_{\mathrm{mix}}^c})\sharp\tilde\mu_m
  \bigr)
\bigr)$,
where $(-\Delta)^{-1}$ denotes the inverse of the negative Laplacian operator with zero Neumann boundary conditions, as in \cite{kim2025sobolevgradientascentoptimal}.
Here $\vf^{(t)}=(f_1^{(t)},\ldots,f_{m-1}^{(t)})$, and
$f_m^{(t)}:=f_{\mathrm{mix}}^{(t)}
=-\sum_{i=1}^{m-1}(\omega_i/\omega_m)f_i^{(t)}$.
Assume that the iterates are continuous, the Sobolev gradients
are well defined, and
$\|\vf^{(0)}-\vf_*\|_{\mathcal H_\Omega}\le R$
and $\|\mathbf g_t\|_{\mathcal H_\Omega}\le M$
for deterministic constants $R,M>0$ independent of $n$ and $T$. 
In Sobolev norms and inner products, the Kantorovich potential $\vf_*$ associated with $(\mu_1,\dots,\mu_m)$ is identified
with its first $m-1$ components.

To recover a primal estimator from the dual iterates, we use an
averaging construction motivated by the primal recovery methods
of \citet{nedic2009approximate}. Their approach obtains approximate
primal solutions by averaging the primal responses generated
during dual subgradient optimization. In our setting, each dual
iterate induces a reconstructed measure for each marginal.
We form an ordinary mixture of these measures across both
marginals and iterations, using the marginal weights and
step sizes, respectively.  
Set $A_T:=\sum_{t=0}^{T-1}\eta_t$ and
$\nu_i^{(t)}:=T_{(f_i^{(t)})^c}\sharp\tilde\mu_i$.
Define the approximate estimator by
\begin{equation}\label{eqn:sga_averaged_estimator}
\tilde\mu_{m,T}^{(n)}
:=
\frac{1}{A_T}
\sum_{t=0}^{T-1}\eta_t
\sum_{i=1}^m\omega_i\nu_i^{(t)}.
\end{equation}
This is an ordinary mixture of the reconstructed measures,
which can be accumulated during the SGA iterations.
Then
\begin{equation}\label{eqn:sga_statistical_bound}
\mathbb E W_2^2(\tilde\mu_{m,T}^{(n)},\bar\mu_m)
\le
C r_n
+
\frac{2}{\alpha A_T}
\left(
R^2+M^2\sum_{t=0}^{T-1}\eta_t^2
\right),
\end{equation}
where $C$ is independent of $n$ and $T$. For a prescribed iteration budget $T$, choosing the constant
step size $\eta_t=\eta_0/\sqrt{T}$ with $\eta_0>0$ gives
$
\mathbb E W_2^2(\tilde\mu_{m,T}^{(n)},\bar\mu_m)
\le
C r_n
+
2/ (\alpha\sqrt{T})
(R^2 / \eta_0 + \eta_0M^2 )$. With the annealing step size $\eta_t=\eta_0/\sqrt{t+1}$,
the optimization term is
$\mathcal O((1+\log T)/\sqrt{T})$.
Consequently, an optimization contribution of order $\varepsilon$ to the squared Wasserstein error bound \eqref{eqn:sga_statistical_bound} is achieved within $\mathcal O(\varepsilon^{-2})$ iterations in the first case
and $\mathcal O(\varepsilon^{-2}\log^2(1/\varepsilon))$
iterations in the second case.
Choosing $\varepsilon$ of order $r_n$ preserves the statistical
convergence rate. The proof of \eqref{eqn:sga_statistical_bound} is provided in Appendix \ref{appendix:deferred_proofs}.
\end{remark}

\begin{proof}[Proof of Theorem \ref{thm_sample_complexity}]
For every \(\nu\in\mathcal P_2(\Omega)\), triangle inequality gives
\[
\begin{aligned}
\left|
\frac12W_2^2(\nu,\mu_j)
-\frac12W_2^2(\nu,\tilde\mu_j)
\right|
&\leq
\frac12\bigl(W_2(\nu,\mu_j)+W_2(\nu,\tilde\mu_j)\bigr)
W_2(\mu_j,\tilde\mu_j)\leq D\,W_2(\mu_j,\tilde\mu_j),
\end{aligned}
\]
where $D$ denotes the diameter of $\Omega$. Taking weighted sums and then infimum over \(\nu\in\cP_2(\Omega)\) gives
\[
|\sE (\tilde\mu_1,\dots,\tilde\mu_m) - \sE (\mu_1,\dots,\mu_m)|
\leq D\sum_{j=1}^m\omega_jW_2(\mu_j,\tilde\mu_j).
\]
Since \(\sum_j\omega_j=1\), Cauchy-Schwarz inequality implies that 
\[
\Bigl(
\sE (\tilde\mu_1,\dots,\tilde\mu_m) - \sE (\mu_1,\dots,\mu_m)
\Bigr)^2
\leq D^2\sum_{j=1}^m\omega_jW_2^2(\mu_j,\tilde\mu_j).
\]
Combining this with Lemma \ref{Lemma:jonathan_smooth_dentisy_estimation} completes the proof. 
\end{proof}

\begin{proof}[Proof of Theorem \ref{thm_W_1_discrepancy}]
Triangle inequality and Cauchy-Schwarz inequality give that
\begin{align}\label{eqn:barycenter_estimation_term_0}
W_2(\tilde{\mu}_{m}^{(n)},\bar\mu_m)
 &\le
 \biggl(\sum_j\omega_jW_2^2( T_{f^c_{*,j}}\sharp \tilde{\mu}_j,\bar\mu_m)\biggr)^{1/2}
 +
 \biggl(\sum_j\omega_j
 W_2^2(\tilde{\mu}_{m}^{(n)},T_{f^c_{*,j}}\sharp \tilde{\mu}_j )\biggr)^{1/2}.
\end{align}
For the first term, the \(1/\alpha\)-Lipschitz continuity of \(T_{f_{*,j}^c}\) as shown in Lemma \ref{lemma:population_maps_change_of_variables}, together with \eqref{primal_dual_relation}, yields
\begin{align} \label{eqn:barycenter_estimation_term_1}
\sum_{j=1}^m
\omega_jW_2^2(T_{f^c_{*,j}}\sharp \tilde{\mu}_j,\bar\mu_m)
=
\sum_{j=1}^m
\omega_jW_2^2(T_{f^c_{*,j}}\sharp \tilde{\mu}_j,T_{f^c_{*,j}}\sharp \mu_j)
 &\le\alpha^{-2}
\sum_{j=1}^m
\omega_jW_2^2(\tilde\mu_j,\mu_j).
\end{align}
Pick an optimal coupling
\(\pi_j\in\Pi(\tilde{\mu}_{m}^{(n)},\tilde\mu_j)\) and then $(\operatorname{id}, T_{f^c_{*,j}}) \sharp \pi_j\in \Pi(\tilde{\mu}_{m}^{(n)},T_{f^c_{*,j}}\sharp\tilde\mu_j)$. Hence,
\begin{align}\label{eqn:W_2_smoothed_barycenter_pushforward_smoothed_input}
W_2^2(\tilde{\mu}_{m}^{(n)},T_{f^c_{*,j}}\sharp\tilde\mu_j)
\le
\int_{\Omega\times\Omega}
\|x-T_{f_{*,j}^c}(y)\|^2\,d\pi_j(x,y).
\end{align}
Denote  $ G_{f_{*,j}}(x,y) := c(x,y)-f_{*,j}(x)-f_{*,j}^c(y)$ and 
then Lemma \ref{lemma:G_f_upper_lower_quadratic-bound} implies that 
\begin{align}\label{eqn:W_2_upper_bounded_by_dual_functional_difference}
 \frac\alpha2\sum_j
 &\omega_j
W_2^2(\tilde{\mu}_{m}^{(n)},T_{f^c_{*,j}}\sharp \tilde{\mu}_j )
\le\sum_j\omega_j\int G_{f_{*,j}}\,d\pi_j
 \nonumber\\
 &=
\sum_{j=1}^{m}\omega_j\int c(x,y)\,d\pi_j
-\sum_{j=1}^{m}\omega_j\int f_{*,j}(x)\,d\pi_j
-\sum_{j=1}^{m}\omega_j\int f_{*,j}^c(y)\,d\pi_j\nonumber\\
&=
\frac12\sum_{j=1}^{m}\omega_j
W_2^2(\tilde{\mu}_{m}^{(n)},\tilde\mu_j)
-\int_\Omega\biggl(\sum_{j=1}^{m}\omega_jf_{*,j}(x)\biggr)
\,d\tilde{\mu}_{m}^{(n)}(x)
-\sum_{j=1}^{m}\omega_j
\int_\Omega f_{*,j}^c(y)\,d\tilde\mu_j(y)\nonumber\\
 &=\sD_{[\tilde\mu_1,\dots,\tilde\mu_m]}(\tilde\vf_*) - \sD_{[\tilde\mu_1,\dots,\tilde\mu_m]}(\vf_*).
\end{align}
The density bounds imply
$\frac LU \mu_j\le\tilde\mu_j\le \frac UL \mu_j$. At the same time, note that 
$T_{f^c_{*,j}}\sharp\mu_j = \bar\mu_m$ and $T_{f_{*,j}}\sharp\bar\mu_m = \mu_j$ for $j\in[m]$,
and the barycenter $\bar\mu_m$ satisfies
$L\alpha^d\le\bar{\mu}_m(x)\le U$
for a.e. $x\in \cK$ by Lemma \ref{lemma:population_maps_change_of_variables},
\[
\frac{L^2\alpha^d}{U}
\le
\frac LU \bar\mu_m
\le T_{f^c_{*,j}}\sharp\tilde\mu_j
\le
\frac UL \bar\mu_m\leq \frac{U^2}{L} \quad \text{for a.e.}\quad x\in \cK.
\]
Then it follows that 
\begin{align}\label{eqn:bound_on_dual_difference_gradient_norm}
\sD_{[\tilde\mu_1,\dots,\tilde\mu_m]}
&(\tilde\vf_*) - \sD_{[\tilde\mu_1,\dots,\tilde\mu_m]}(\vf_*) 
\lesssim
\sum_{i=1}^{m-1} \omega_i \|T_{f^c_{*,i}}\sharp\tilde\mu_i-T_{f^c_{*,\mathrm{mix}}} \sharp\tilde\mu_m\|^2_{\dot{\bH}^{-1}} \nonumber\\
&=
\sum_{i=1}^{m-1} \omega_i \|T_{f^c_{*,i}}\sharp\tilde\mu_i -
T_{f^c_{*,i}}\sharp\mu_i
+
T_{f^c_{*,\mathrm{mix}}} \sharp\mu_m
-
T_{f^c_{*,\mathrm{mix}}} \sharp\tilde\mu_m\|^2_{\dot{\bH}^{-1}} \nonumber\\
&\lesssim
\sum_{i=1}^{m-1} 
\omega_i 
\Bigl(
\|T_{f^c_{*,i}}\sharp\tilde\mu_i -
T_{f^c_{*,i}}\sharp\mu_i \|^2_{\dot{\bH}^{-1}} 
+
\|T_{f^c_{*,\mathrm{mix}}} \sharp\mu_m
-
T_{f^c_{*,\mathrm{mix}}} \sharp\tilde\mu_m\|^2_{\dot{\bH}^{-1}}
\Bigr)
\nonumber\\
&\lesssim
\sum_{i=1}^{m-1} 
\omega_i 
\Bigl(
W_2^2 (\tilde\mu_i, 
\mu_i) 
+
W_2^2 (\mu_m
,
\tilde\mu_m)
\Bigr).
\end{align}
The first inequality follows from Lemma \ref{Lemma:Strong_concavity}, and the equality follows from the barycenter characterization \eqref{primal_dual_relation}. The second inequality uses \(\|a+b\|^2\le 2\|a\|^2+2\|b\|^2\). The third inequality follows from Lemma \ref{lemma:W_2H_-1_bound} and Lemma \ref{lemma:population_pushforward_Hminus1_stability}.
As a consequence,
\begin{align} \label{eqn:barycenter_estimation_term_2}
\sum_{j=1}^m\omega_j
W_2^2(\tilde{\mu}_{m}^{(n)},T_{f^c_{*,j}}\sharp \tilde{\mu}_j )
\lesssim
\frac2\alpha
\sum_{i=1}^{m-1} 
\omega_i 
\Bigl(
W_2^2 (\tilde\mu_i, 
\mu_i) 
+
W_2^2 (\mu_m
,
\tilde\mu_m)
\Bigr).
\end{align}
Combining \eqref{eqn:barycenter_estimation_term_0}, \eqref{eqn:barycenter_estimation_term_1}, \eqref{eqn:barycenter_estimation_term_2} and Lemma \ref{Lemma:jonathan_smooth_dentisy_estimation} concludes the proof.
\end{proof}

\section{Application: sample complexity of two-layer sampling model}\label{Sec_two_staged_sampling}


In many applications of distributional data analysis, the fundamental analytical unit is a probability distribution rather than an individual observation. Representative examples include regression analysis with distributional predictor or response \citep{chen2023wasserstein, 10.1093/jrsssb/qkad051,zhu2024spherical,Alberto_two_staged_sampling25}, and distributional principal component analysis
\citep{bigot2017geodesic, cazelles2018geodesic}. In these settings, each data object is naturally modeled as a probability measure $\mu_j \in \mathcal{P}_2(\mathbb{R}^d)$, and statistical analysis is performed directly at the level of distributions.

In practice, however, the distributions $\mu_1, \dots, \mu_m$ are not directly observed. Instead, for each $j \in [m]$, one only has access to an i.i.d.\ sample $X_{j1}, \dots, X_{jn} \sim \mu_j$. This leads to a two-layer sampling scheme. At the first layer, the latent distributions $\mu_1, \dots, \mu_m$ are independently drawn from a population distribution $\mathbb{P} \in \mathcal{P}_2(\mathcal{P}_2(\mathbb{R}^d))$, but remain unobserved. At the second layer, conditional on each $\mu_j$, an i.i.d.\ sample $\{X_{ji}\}_{i=1}^n$ is observed. The full dataset thus consists only of the second-layer samples $\{X_{ji} : j \in [m],\, i \in [n]\}$.

Our results enable theoretical guarantees for this widely used two-layer sampling framework, explicitly accounting for the randomness arising from both the first-layer sampling of distributions and the second-layer sampling of observations. In particular, under uniform barycentric weights $\omega_j = 1/m$, we could establish the statistical convergence rate of the smoothed empirical barycenter $\tilde{\mu}_m^{(n)}$ toward the population barycenter $\mu^*$ associated with $\bP$. This result follows by combining our second-layer sample complexity analysis with existing guarantees for the first-layer sampling of distributions \citep{Thibaut_Quentin_Philippe_Austin_22fast}, under suitable regularity conditions. To this end, we adopt the following assumption, introduced in \cite{Thibaut_Quentin_Philippe_Austin_22fast}, which ensures a parametric convergence rate for the first-layer sampling.

\begin{assumption}\label{assumption_Thibaut_22_fast}
    Every $\mu\in \text{supp}(\bP)$ is the pushforward of $\mu^*$ by the gradient of a $\kappa$-strongly convex and $\lambda$-smooth function and $\lambda-\kappa<1$.
\end{assumption}

\noindent
We note that the above Assumption \ref{assumption_Thibaut_22_fast} used in \cite{Thibaut_Quentin_Philippe_Austin_22fast} is heuristically stronger than our Assumption~\ref{assumption_Wasserstein_potential}, which is imposed for the second-layer analysis. 
In particular, the former enforces a specific relationship between the strong convexity and smoothness parameters, whereas the latter only requires a strong convexity parameter. 
Under Assumptions~\ref{assumption_upper_lower_bounded_density}--\ref{assumption_Thibaut_22_fast}, we can unify the results for both sampling layers in two-layer sampling model and obtain the following bound.

\begin{theorem}[Two-layer sampling model sample complexity]\label{sample_complexity_two_staged_model} Under Assumptions \ref{assumption_upper_lower_bounded_density}--\ref{assumption_Thibaut_22_fast}, 

\begin{align*}
        \bE W_2^2(\tilde{\mu}_m^{(n)}, \mu^*) \lesssim
\begin{cases}
 m^{-1} + n^{-\frac{2+2s}{d+2s}}, & d \ge 3, \\[0.6em]
 m^{-1} + n^{-1}(\log n)^2,       & d = 2, \\[0.4em]
 m^{-1} + n^{-1},             & d = 1 .
\end{cases}
\end{align*}
Here, $\lesssim$ hides constants depending on $\lambda,\kappa,\bP,\alpha,\ell,s,d,L,U$.
\end{theorem}

\begin{remark}
Under milder conditions, the state-of-the-art \cite{Carlier_barycenter_24PTRF} combined with Corollary 4.4 in \cite{Thibaut_Quentin_Philippe_Austin_22fast} 
implies that 
\begin{align*}
\bE W_2^2(\hat{\mu}_m^{(n)}, \mu^*) \lesssim
\begin{cases}
 m^{-1} + n^{-1/3d}, & d \geq 5, \\[0.6em]
 m^{-1} + n^{-1/12}(\log n)^{1/6},       & d = 4, \\[0.4em]
 m^{-1} + n^{-1/12},             & d \leq 3 .
\end{cases}
\end{align*}
\end{remark}

\begin{proof}[Proof of Theorem \ref{sample_complexity_two_staged_model}.]
Decompose $W_2^2(\tilde{\mu}_m^{(n)}, \mu^*)\leq 2W_2^2(\tilde{\mu}_m^{(n)}, \bar{\mu}_m) + 2W_2^2( \bar{\mu}_m, \mu^*)$ and the proof then follows directly from Corollary 4.4 in \cite{Thibaut_Quentin_Philippe_Austin_22fast} and Theorem \ref{thm_W_1_discrepancy}.    
\end{proof}

\section{Numerical Experiments}\label{sec:numerical_experiment}
We compare the empirical plug-in and smoothness-aware barycenter estimators across dimensions $d=1,2,3$. For each dimension, we construct $m=3$ probability measures supported on $\Omega=[0,1]^d$ whose equally weighted Wasserstein barycenter
\(
\bar\mu_m
=
\arg\min_{\nu\in\mathcal P_2(\Omega)}
\frac{1}{m}\sum_{j=1}^{m}W_2^2(\nu,\mu_j)
\)
is explicitly known, providing a benchmark for comparing the accuracy and empirical convergence rates of the two estimators.
Specifically, we fix a
product reference $\rho(x)=\prod_{k=1}^{d}\rho_0(x_k)$, where 
$\rho_0(t)\propto\exp\bigl(-|t-\tfrac12|^{5/2}\bigr)$ (in $\bH^s$ only for $s<3$) bounded away from zero on
$[0,1]$. The marginals are pushforwards $\mu_j=(T_j)\sharp\rho$ with $T_j=\mathrm{id}+\nabla h_j$, where
$h_j=a_j\psi_a+b_j\psi_b$, $\psi_a(x)=-\pi^{-2}\sum_{k}\cos(\pi x_k)$, $\psi_b(x)=\pi^{-2}\prod_{k}\cos(\pi x_k)$,
$(a_1,a_2,a_3)=(0.45,-0.15,-0.30)$ and $(b_1,b_2,b_3)=(-0.25,0.40,-0.15)$. Each $\varphi_j(x)=\tfrac12\|x\|^2+h_j(x)$ is convex and then $T_j$ is the optimal transport map from $\rho$ to $\mu_j$. 
The first-order characterization of barycenters \cite[Proposition~3.8]{AguehCarlier2011} implies that $\bar\mu_m=\rho$.

For each sample size $n\in\{128,256,512,1024\}$ per marginal and each replicate ($200$ replicates for $d=1$ and $20$ for $d\in\{2,3\}$), we draw $n$ independent points from each $\mu_j$.
From the same samples we compute the empirical plug-in barycenter estimator $\hat\mu^{(n)}_m$ and the smoothness-aware barycenter estimator $\tilde\mu^{(n)}_m$. 
The empirical plug-in estimator is the barycenter of the empirical distributions of each marginal distributions, obtained from the multi-marginal linear program of \citet{anderes2016discrete} for $d\in\{2,3\}$ and from averaging quantile functions for $d=1$.

The smoothness-aware estimator (Algorithm~\ref{alg:main}) is constructed in two steps: each marginal is first replaced by a smooth density estimate, and the barycenter of these estimates is then computed on a common grid.
For the first step, each marginal density is estimated by a degree-$2$ Alpert multiwavelet series
\citep{alpert1993class} up to resolution level $J$, whose coefficients are estimated from the sample, giving a
piecewise-quadratic function on $2^{J+1}$ dyadic cells per axis. We take $J=1$ throughout. This agrees with the theoretical guarantee level $J=\lfloor \log_2 n/(2s+d)\rfloor$ used in the proof of
Lemma~\ref{Lemma:jonathan_smooth_dentisy_estimation}, with $s=2$, for every $(d,n)$ considered except $d=1$, $n=1024$. The resulting piecewise-quadratic estimate is averaged over the cells of a computational grid with $G$ cells per axis
($G=8192$ for $d=1$, $G=256$ for $d=2$ and $G=80$ for $d=3$) and then projected in $\dot\bH^{-1}$ onto the set of nonnegative densities, which yields a smoothed estimate of each marginal on the grid.
For the second step, when $d\in\{2,3\}$ the barycenter of the three smoothed marginals is computed by Sobolev gradient ascent on the semi-dual \citep{kim2025sobolevgradientascentoptimal}. When $d=1$ the barycenter is computed exactly by averaging quantile functions.

Figure~\ref{fig:abs52} compares the empirical plug-in estimator $\hat\mu_m^{(n)}$ (orange) with the smoothness-aware estimator $\tilde\mu_m^{(n)}$ (blue). For each dimension, we plot the squared Wasserstein error $W_2^2(\cdot,\rho)$, averaged over replicates, against the sample size $n$ on log--log axes. 
For $d=1$, the two estimators exhibit comparable errors, and both decay at the fitted rate $n^{-1.06}$.
For $d=2$, the smoothness-aware estimator yields smaller errors, while the fitted decay rates remain similar ($n^{-0.95}$ versus $n^{-0.92}$). 
For $d=3$, smoothed-aware estimator achieves both smaller errors and faster fitted decay rate than the empirical plug-in estimator ($n^{-0.86}$ versus $n^{-0.70}$) over the sample sizes considered. In addition, its fitted decay exponent of $0.86$ is consistent with the theoretical exponent $6/7$ in Theorem~\ref{thm_W_1_discrepancy} for $s=2$.

{\color{purple}
}

\begin{figure}[!htbp]
    \centering
    \begin{minipage}{0.99\textwidth}
        \centering

        \begin{subfigure}[t]{0.33\linewidth}
            \centering
            \includegraphics[width=\linewidth]
            {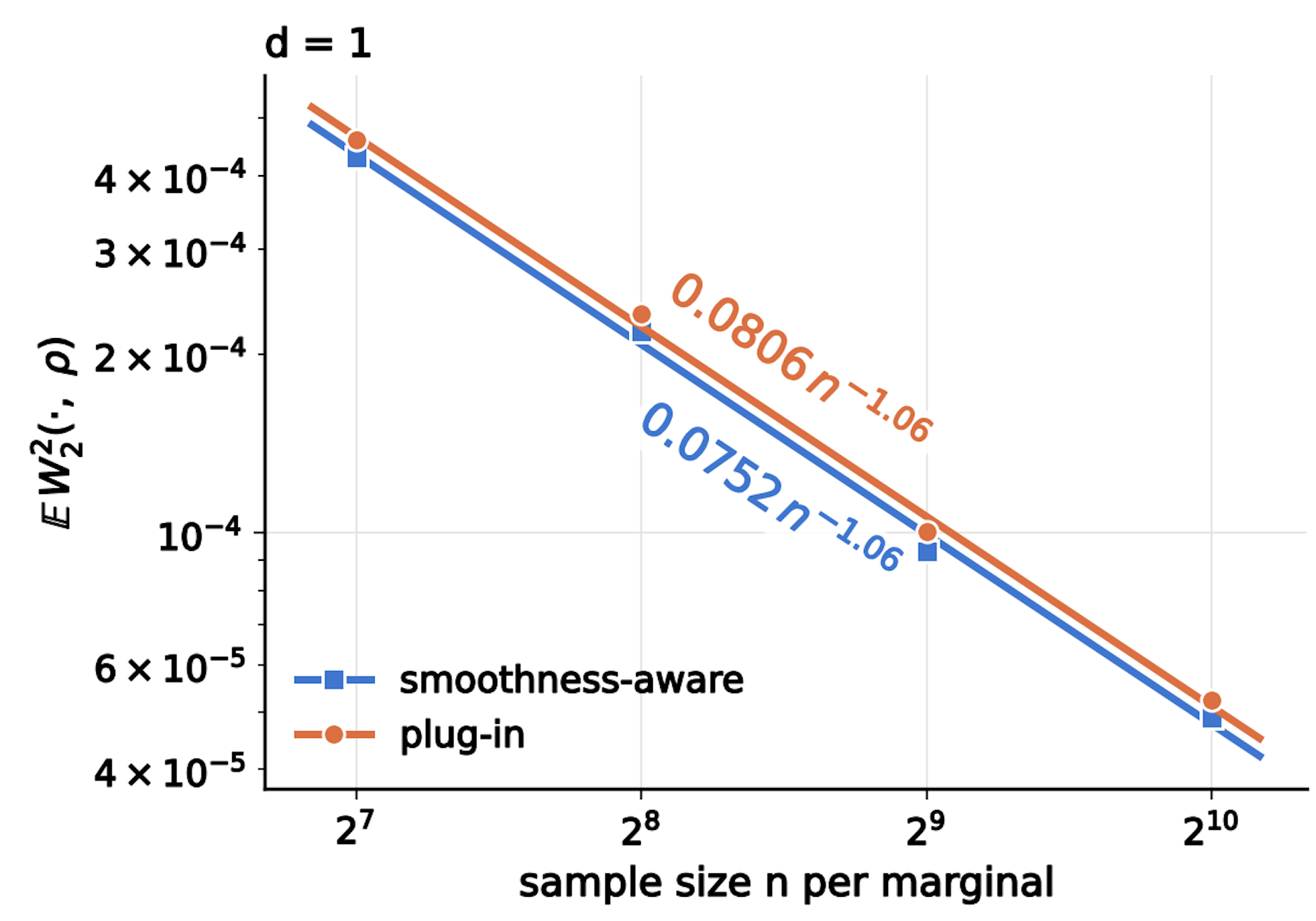}
            \caption{$d=1$}
        \end{subfigure}
        \hfill
        \begin{subfigure}[t]{0.33\linewidth}
            \centering
            \includegraphics[width=\linewidth]
            {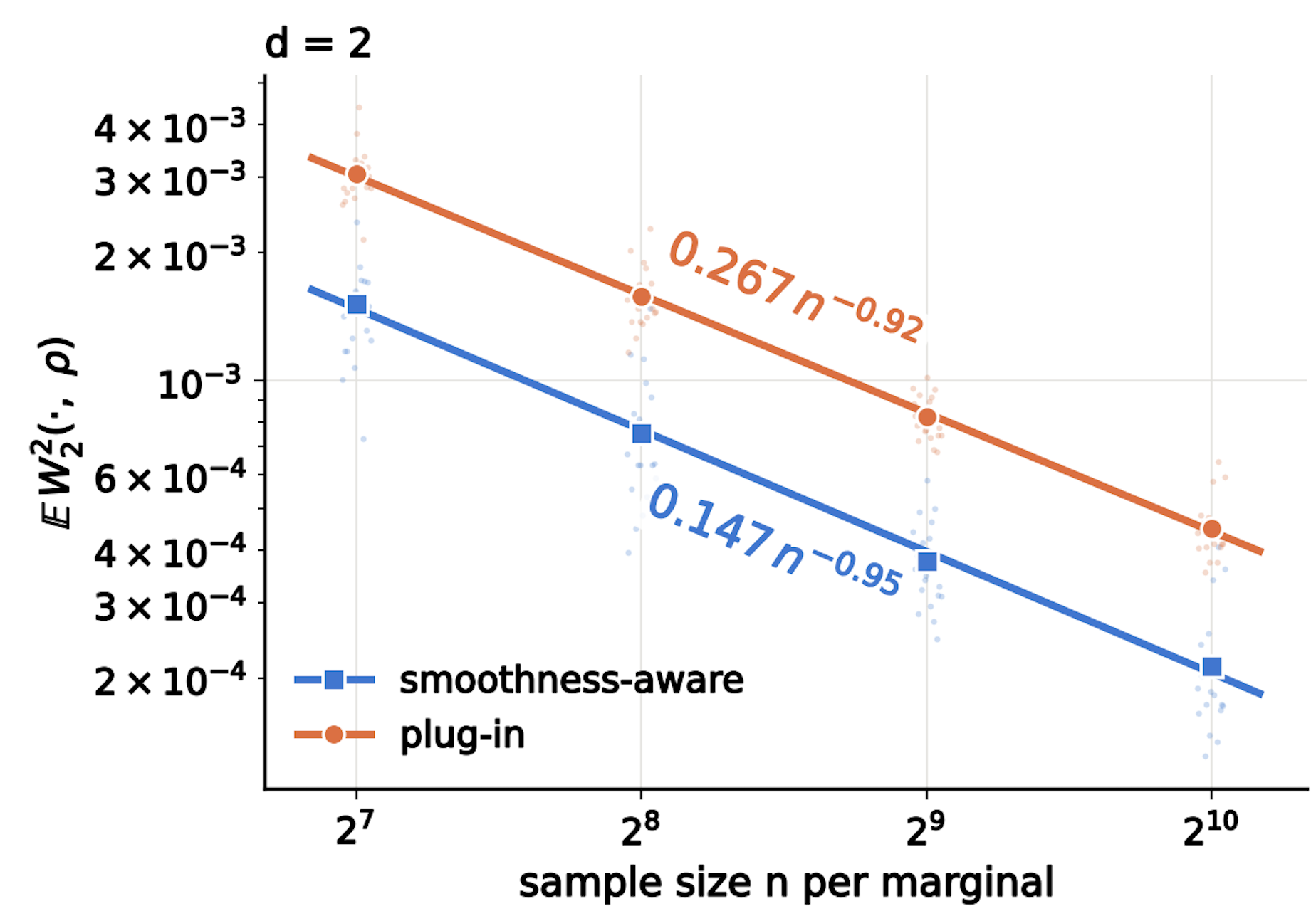}
            \caption{$d=2$}
        \end{subfigure}
        \hfill
        \begin{subfigure}[t]{0.33\linewidth}
            \centering
            \includegraphics[width=\linewidth]
            {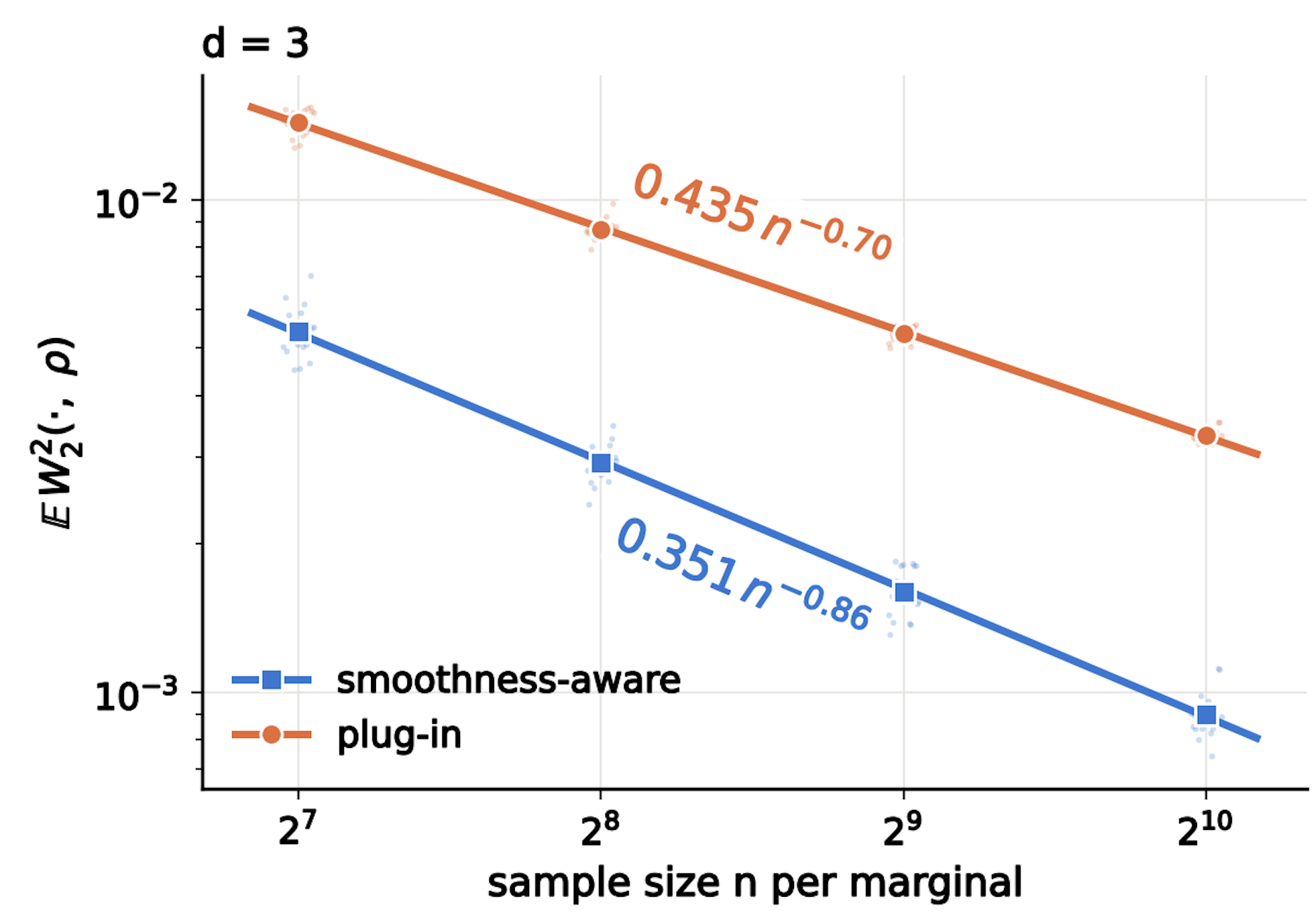}
            \caption{$d=3$}
        \end{subfigure}

    \end{minipage}

    \caption{
        Empirical convergence rates of the barycenter estimators in dimensions $d=1,2,3$. Each panel plots the squared Wasserstein error $W_2^2(\cdot,\rho)$ against the sample size per marginal, $n=2^7,\ldots,2^{10}$. 
        For $d=2,3$, faint dots are the errors of the individual replicates (offset slightly horizontally for visibility) and
  large markers are the means over replicates ($200$ for $d=1$, $20$ for $d=2,3$). The lines show power-law fits $Cn^\beta$ obtained by regressing the log mean error on $\log n$ by least squares.
    }
    \label{fig:abs52}
\end{figure}

\FloatBarrier

\acks{X. Chen acknowledges support from NSF DMS-2413404 and an unrestricted gift from the Simons Foundation. C. Zhu acknowledges support from NSF DMS-2412832. The authors used AI tools to proofread the manuscript and improve the clarity of its language, to help check the mathematical proofs, and to assist in developing the code used in the study. The authors take full responsibility for the final manuscript, including its mathematical results, proofs, and code.}

\bibliography{unreg_barycenter}

\begin{appendix}

\section{Deferred Proofs}\label{appendix:deferred_proofs}

\begin{proof}[Proof of Proposition \ref{proposition_Norm_dual_objective_gradient}.]
The Cauchy-Schwarz inequality and definition give that
\begin{align*}
\bigl \langle \nabla \sD_{[\tilde\mu_1,\dots,\tilde\mu_m]} (\vf_*), &\,(\varphi_1, \ldots, \varphi_{m-1}) \bigr\rangle
= \sum_{i=1}^{m-1} \omega_i \int \varphi_i 
\, d \, ( T_{f_{*,\mathrm{mix}}^c}\sharp \tilde\mu_m - T_{f_{*,i}^c}\sharp\tilde\mu_i )  \\
&\leq \sum_{i=1}^{m-1} \omega_i \|\varphi_i\|_{\dot{\bH}^1}
\|T_{f_{*,i}^c}\sharp\tilde\mu_i -  T_{f_{*,\mathrm{mix}}^c}\sharp \tilde\mu_m\|_{\dot{\bH}^{-1}}\\
&\leq \left ( \sum_{i=1}^{m-1} \omega_i \|\varphi_i\|^2_{\dot{\bH}^1} \right )^{1/2}  \left( \sum_{i=1}^{m-1} \omega_i  \|T_{f_{*,i}^c}\sharp\tilde\mu_i -  T_{f_{*,\mathrm{mix}}^c}\sharp \tilde\mu_m\|^2_{\dot{\bH}^{-1}} \right)^{1/2}\\
&=\| (\varphi_1,\dots,\varphi_{m-1}) \|_{\sH_{m-1}} \left( \sum_{i=1}^{m-1} \omega_i  \|T_{f_{*,i}^c}\sharp\tilde\mu_i -  T_{f_{*,\mathrm{mix}}^c}\sharp \tilde\mu_m\|^2_{\dot{\bH}^{-1}} \right)^{1/2}.
\end{align*}
Here, the equality is attainable for $\varphi_i = \frac{ (-\Delta)^{-1} \left (   T_{f_{*,\mathrm{mix}}^c}\sharp \tilde\mu_m -  T_{f_{*,i}^c}\sharp\tilde\mu_i\right ) }{( \sum_{j=1}^{m-1} \omega_j \|  
T_{f_{*,j}^c}\sharp\tilde\mu_j -  T_{f_{*,\mathrm{mix}}^c}\sharp \tilde\mu_m
\|^2_{\dot{\bH}^{-1} } )^{1/2}}$ where $(-\Delta)^{-1}$ denotes the inverse of the negative Laplacian operator with zero Neumann boundary conditions.
\end{proof}

\begin{proof}[Proof of Lemma \ref{Lemma:Strong_concavity}]
Put $\rho_j = T_{f_{*,j}}\sharp
\bigl(
T_{f^c_{*,m}}\sharp\tilde\mu_m
\bigr)$. The density bounds imply
$\frac LU \mu_m\le\tilde\mu_m\le \frac UL \mu_m$. 
Recall from \eqref{primal_dual_relation} that
$T_{f^c_{*,j}}\sharp\mu_j=\bar\mu_m$ and
$T_{f_{*,j}}\sharp\bar\mu_m=\mu_j$ for all $j\in[m]$. Hence,
\[
\frac{L}{U}\bar\mu_m
\leq T_{f^c_{*,m}}\sharp\tilde\mu_m,\, T_{f^c_{*,j}}\sharp\tilde\mu_j
\leq \frac{U}{L}\bar\mu_m,
\qquad
\frac{L}{U}\mu_j
\leq \rho_j
\leq \frac{U}{L}\mu_j,
\quad j\in[m].
\]
Note that the barycenter characterization \eqref{primal_dual_relation} gives
\(
 \rho_{j}-\tilde\mu_j
 =T_{f_{*,j}}\sharp(
 T_{f^c_{*,m}}\sharp\tilde\mu_m
 -T_{f^c_{*,j}}\sharp\tilde\mu_j),
\)
and we have
\begin{align}\label{eqn:W_2_Sobolev_bound}
W_2^2(\rho_{j},\tilde\mu_j)
\lesssim\,
\bigl\|
T_{f_{*,j}}\sharp
\bigl(
T_{f^c_{*,m}}\sharp\tilde\mu_m
\bigr)
-
T_{f_{*,j}}\sharp
\bigl(
T_{f^c_{*,j}}\sharp\tilde\mu_j)
\bigr\|^2_{\dot\bH^{-1}(\Omega)}
\lesssim\,
\bigl\|{T_{f^c_{*,m}}\sharp\tilde\mu_m-T_{f^c_{*,j}}\sharp\tilde\mu_j}
\bigr\|_{\dot\bH^{-1}(\cK)}^2.
\end{align}
Here, the first (resp. second) inequality follows from 
Lemma \ref{lemma:W_2H_-1_bound}  (resp. Lemma \ref{lemma:population_pushforward_Hminus1_stability}). 
Choose optimal couplings
\(\gamma_j\in\Pi(\rho_{j},\tilde\mu_j)\), and then the coupling \( \pi_j :=(T_{f_{*,j}^c},\operatorname{id}) \sharp \gamma_j \in\Pi (T_{f_{*,m}^c}\sharp\tilde\mu_{m},\tilde\mu_j) \) because $T_{f^c_{*,j}}\sharp\rho_j = T_{f^c_{*,j}}\sharp 
\Bigl(T_{f_{*,j}}\sharp
\bigl(
T_{f^c_{*,m}}\sharp\tilde\mu_m
\bigr)\Bigr)
=
T_{f^c_{*,m}}\sharp\tilde\mu_m$.
It follows from the definition of the optimal transport cost that
\[
\frac12 W_2^2 
\bigl( T_{f_{*,m}^c}\sharp\tilde\mu_m,\tilde\mu_j
\bigr)
\le
\int_{\Omega\times\Omega}
c(T_{f_{*,j}^c}(u),v)\,d\gamma_j(u,v).
\]
Then strong duality \eqref{equation_strong_duality} gives
\begin{align}\label{eqn:gap_bound_optimum}
\sD_{[\tilde\mu_1,\dots,\tilde\mu_m]}(\tilde\vf_*) 
&- \sD_{[\tilde\mu_1,\dots,\tilde\mu_m]}(\vf_*)
=
\sE (\tilde\mu_1,\dots,\tilde\mu_m)- \sD_{[\tilde\mu_1,\dots,\tilde\mu_m]}(\vf_*)
\nonumber
\\
&\le
\sum_{j=1}^m\omega_j
\int_{\Omega\times\Omega}
c(T_{f^c_{*,j}}(u),v)\,d\gamma_j(u,v)
-
\sum_{j=1}^m\omega_j
\int_\Omega f_{*,j}^c(v)\,d\tilde\mu_j(v)
\nonumber
\\
&=
\sum_{j=1}^m\omega_j
\int_{\Omega\times\Omega}
\Bigl(
c(T_{f^c_{*,j}}(u),v)-f_{*,j}^c(v)
\Bigr)
\,d\gamma_j(u,v).
\end{align}
Note that $c(T_{f^c_{*,j}}(u),v)-f_{*,j}^c(v)
=
G_{f_{*,j}}(T_{f^c_{*,j}}(u),v)+f_{*,j}(T_{f^c_{*,j}}(u))$
for 
$
G_{f_{*,j}}(x,y)
:=
c(x,y)-f_{*,j}(x)-f_{*,j}^c(y).
$ 
Since the first marginal of $\pi_j$ is $T_{f^c_{*,m}}\sharp\tilde\mu_m$, 
\begin{align}\label{eqn:gap_bound_optimum'}
\eqref{eqn:gap_bound_optimum}
&=
\sum_{j=1}^m\omega_j
\int_{\Omega\times\Omega}
G_{f_{*,j}}(T_{f^c_{*,j}}(u),v)+f_{*,j}(T_{f^c_{*,j}}(u))
\,d\gamma_j(u,v)
\nonumber\\
&=
\sum_{j=1}^m\omega_j
\int_{\Omega\times\Omega}
G_{f_{*,j}}(T_{f^c_{*,j}}(u),v)
\,d\gamma_j(u,v)
+
\int_{\Omega} 
\Bigl(
\sum_{j=1}^m\omega_j f_{*,j}(u)
\Bigr)
\,d(T_{f^c_{*,m}}\sharp \tilde{\mu}_m)(u)\nonumber\\
&=
\sum_{j=1}^m\omega_j
\int_{\Omega\times\Omega}
G_{f_{*,j}}(T_{f^c_{*,j}}(u),v)
\,d\gamma_j(u,v).
\end{align}
Thus, a combination of \eqref{eqn:gap_bound_optimum}, \eqref{eqn:gap_bound_optimum'} and Lemma \ref{lemma:G_f_upper_lower_quadratic-bound} implies that
\begin{align}\label{eqn:gap_bound_W_2}
\sD_{[\tilde\mu_1,\dots,\tilde\mu_m]}(\tilde\vf_*) - \sD_{[\tilde\mu_1,\dots,\tilde\mu_m]}(\vf_*)
&\le
\sum_{j=1}^m\omega_j
\int G_{f_{*,j}}(T_{f^c_{*,j}}(u),v)\,d\gamma_j(u,v)
\nonumber\\
&\le
\frac1{2\alpha}
\sum_{j=1}^m\omega_j
\int \|u-v\|^2\,d\gamma_j(u,v)
=
\frac1{2\alpha}
\sum_{j=1}^m\omega_j
W_2^2(\rho_{j},\tilde\mu_j).
\end{align}
Combining \eqref{eqn:W_2_Sobolev_bound} and \eqref{eqn:gap_bound_W_2} concludes the proof.
\end{proof}

\begin{proof}[Proof of Lemma \ref{Lemma:jonathan_smooth_dentisy_estimation}.]

We first recall the equivalence between the Sobolev norm
$\|\cdot\|_{\bH^s}$ and the Besov norm
$\|\cdot\|_{B^s_{2,2}}$
(see \cite{Niles_smooth_density_minimax22} for the definition
of Besov norms). The remainder of the proof to construct one estimator candidate then follows by a straightforward adaptation of the arguments in \cite{Niles_smooth_density_minimax22}.
Let $q_J$ denote the empirical wavelet estimator at resolution $J=\left\lfloor\frac{\log_2 n}{2s+d}\right\rfloor$  (cf. Section~4.1 of \cite{Niles_smooth_density_minimax22}).
To obtain an estimator with $L \leq \tilde{\mu} \leq U$, in Lemma 8 of \cite{Niles_smooth_density_minimax22}, one projects the wavelet estimator $q_J$
to $\mathcal D_{L,U}$:
\[
\tilde\mu_n\in
\operatorname*{argmin}_{g\in\mathcal D_{L,U}}
\|g-q_J\|_{B^{-1}_{2,1}},
\qquad
\mathcal D_{L,U}:=
\left\{
g\in L^2(\Omega),
\int_\Omega g\,dx=1,\quad
L\le g\le U
\right\}.
\]
Since $\mu\in\mathcal D_{L,U}$, the projection argument of
\cite[Lemma~8]{Niles_smooth_density_minimax22} gives
\[
W_2(\mu,\tilde\mu_n)
\lesssim \|\tilde\mu_n-\mu\|_{B^{-1}_{2,1}}
\le
\|\tilde\mu_n-q_J\|_{B^{-1}_{2,1}}
+\|q_J-\mu\|_{B^{-1}_{2,1}}
\le 2\|q_J-\mu\|_{B^{-1}_{2,1}}.
\]
By orthonormality and the bound $\mu\le U$, each empirical
wavelet coefficient has variance at most $U/n$, as in the proof of Proposition~4 \cite{Niles_smooth_density_minimax22}.
Combining this observation with the Sobolev coefficient
bound and Minkowski's inequality in $L^2$, arguing as proof of Theorem 1 in \cite{Niles_smooth_density_minimax22} yields
\[
\left(
\mathbb E W_2^2(\mu,\tilde\mu_n)
\right)^{1/2}
\lesssim
\ell\,2^{-J(s+1)}
+
\sqrt{\frac Un}
\left(
1+\sum_{j=0}^{J}2^{j(d/2-1)}
\right).
\]
The chosen $J$ gives the asserted rates.
Finally, both $L\leq\mu,\tilde\mu_n\leq U$,
so Lemma~\ref{lemma:W_2H_-1_bound} gives
\[
\frac1U
\|\mu-\tilde\mu_n\|_{\dot{\bH}^{-1}}^2
\le
W_2^2(\mu,\tilde\mu_n)
\le
\frac1L
\|\mu-\tilde\mu_n\|_{\dot{\bH}^{-1}}^2.
\]
This completes the proof.
\end{proof}

\begin{proof}[Proof of Proposition \ref{prop_bounded_dual_potentials}.]
Since $f_{*,\mathrm{mix}} = -\sum_{j=1}^{m-1}\frac{\omega_j}{\omega_m} f_{*,j}$, it suffices to 
prove the claim for $i\in[m-1]$. Normalization gives $f_{*,i}\leq 0 $, 
\begin{align*}
    f_{*,i}^c (y) = \inf_{x\in\Omega}\bigl(\frac{1}{2}\|x-y\|^2 - f_{*,i}(x) \bigr) \geq  \inf_{x\in\Omega}\bigl( 0 - f_{*,i}(x) \bigr)=0.
\end{align*}
Also, for $\|c\|_\infty := \sup_{x,y\in\Omega}\frac{1}{2}\|x-y\|^2$,
\begin{align*}
    f_{*,i}^c (y) = \inf_{x\in\Omega}\bigl(\frac{1}{2}\|x-y\|^2 - f_{*,i}(x) \bigr) \leq  \inf_{x\in\Omega}\bigl( \|c\|_\infty - f_{*,i}(x) \bigr) = \|c\|_\infty.
\end{align*}
We know by Theorem 4 in \cite{kim2025sobolevgradientascentoptimal} that $f_{*,i}$ is c-concave for all $i\in\{1,\dots,m-1,\mathrm{mix}\}$. Hence, 
\begin{align*}
    f_{*,i} (x)= f_{*,i}^{cc} (x) = \inf_{y\in\Omega}\bigl(\frac{1}{2}\|x-y\|^2 - f^c_{*,i}(y) \bigr) \geq  \inf_{y\in\Omega}\bigl( 0 - f^c_{*,i}(y) \bigr) \geq -\|c\|_\infty.
\end{align*}
So we have 
\begin{align*}
    -\|c\|_\infty \leq f_{*,i} (x) \leq 0. \quad 
\end{align*}
The gradient bound follows from Proposition \ref{proposition_lipschitz_c_transform}. This completes the proof.
\end{proof}

\begin{proof}[Proof of Lemma \ref{lemma:population_maps_change_of_variables}]
By \cite[Theorem~5.1]{AguehCarlier2011}, $\bar\mu_m$ is absolutely continuous. 
As \cite[Proposition 3.5, Remark 3.7]{AguehCarlier2011}, one could argue that 
$T_{f^c_{*,j}} \circ (T_{f_{*,j}}|_{\mathcal K} )= \operatorname{id}_\cK$ and $(T_{f_{*,j}} |_{\mathcal K})
\circ T_{f^c_{*,j}} = \operatorname{id}_\Omega$.
For a.e. \(x\in \cK\), set
$
 A_j(x) := D T_{f_{*,j}}(x) = I -D^2f_{*,j}(x).
$ and we have $\sum_j\omega_jA_j(x)=I$. Change of variable gives
\[
\bar \mu_m(x)
 =
 \mu_j \bigl( T_{f_{*,j}}(x) \bigr)\det A_j(x).
 \]
Note that $ A_j(x)\succeq\alpha I$, one has \(\bar \mu_m\ge L\alpha^d\) a.e. on \(\cK\). Concavity of \(\log\det\) yields
\[
 0
 =
 \log \det \sum_j\omega_jA_j(x)
 \ge\sum_j\omega_j\log\det A_j
 =\log\bar \mu_m-\sum_j\omega_j\log \mu_j(T_{f_{*,j}})
 \ge\log\bar \mu_m-\log U.
\]
Thus \(\bar \mu_m\le U\), and then 
\begin{align}\label{eqn:curvature_det_upper_bound}
\det A_j(x)
=
\frac{\bar \mu_m(x)}{\mu_j(T_{f_{*,j}}(x))}
\le\frac UL \,\,\text{a.e. on}\,\, \cK, \quad \text{for every}\, j\in [m].
\end{align}
Let $\lambda_1(x),\ldots,\lambda_d(x)$ be the eigenvalues of $A_j(x)$ and then  $\lambda_r(x)\ge\alpha$ for every $r\in[d]$ by strong convexity.
For each $r\in[d]$,
\[
\lambda_r(x)
=
\frac{\det A_j(x)}{\prod_{\ell\ne r}\lambda_\ell(x)}
\le
\frac{\det A_j(x)}{\alpha^{d-1}}
\le
\frac{U}{L\alpha^{d-1}}.
\]
Consequently,
\[
\alpha I
\preceq D T_{f_{*,j}}(x)
\preceq \frac{U}{L\alpha^{d-1}} I
\quad\text{for a.e. }x\in \cK,
\]
and differentiation gives
\[
DT_{f^c_{*,j}}(y)=[DT_{f_{*,j}}(T_{f^c_{*,j}}(y))]^{-1},\qquad \frac{L\alpha^{d-1}}{U} I\preceq DT_{f_{*,j}^c}\preceq\alpha^{-1} I 
 \quad\text{a.e. on }\Omega.
\]
\end{proof}

\begin{proof}[Proof of
Remark \ref{thm_W_1_discrepancy_partial}]
The triangle inequality and Cauchy--Schwarz inequality,
with the normalized weights $\omega_j/\theta_J$ for $\theta_J:=\sum_{j\in J}\omega_j>0$ imply that
\begin{align}
 W_2(\tilde\mu_m^{(n)},\bar\mu_m)
 \le
 \Bigl(
   \frac1{\theta_J}\sum_{j\in J}\omega_j
   W_2^2\bigl(
     T_{f_{*,j}^c}\sharp\tilde\mu_j,\bar\mu_m
   \bigr)
 \Bigr)^{1/2}
+
 \Bigl(
   \frac1{\theta_J}\sum_{j\in J}\omega_j
   W_2^2\bigl(
     \tilde\mu_m^{(n)},
     T_{f_{*,j}^c}\sharp\tilde\mu_j
   \bigr)
 \Bigr)^{1/2}.
 \label{eqn:barycenter_estimation_term_0_partial}
\end{align}
For the first term, argue as \eqref{eqn:barycenter_estimation_term_1} gives
\begin{align}
 \frac1{\theta_J}\sum_{j\in J}\omega_j
 W_2^2\bigl(
   T_{f_{*,j}^c}\sharp\tilde\mu_j,\bar\mu_m
 \bigr)
\le
 \frac1{\alpha^2\theta_J}
 \sum_{j\in J}\omega_jW_2^2(\tilde\mu_j,\mu_j).
 \label{eqn:barycenter_estimation_term_1_partial}
\end{align}
For the second term, same argument as \eqref{eqn:W_2_upper_bounded_by_dual_functional_difference} 
\begin{align}
 \frac\alpha2\sum_{j\in J}\omega_j
 W_2^2\bigl(
   \tilde\mu_m^{(n)},
   T_{f_{*,j}^c}\sharp\tilde\mu_j
 \bigr)
\leq
 \sD_{[\tilde\mu_1,\dots,\tilde\mu_m]}(\tilde\vf_*)
 -
 \sD_{[\tilde\mu_1,\dots,\tilde\mu_m]}(\vf_*).
 \label{eqn:partial_gap_transport}
\end{align}
For every $j\in[m]$ and $y,z\in\Omega$, Proposition \ref{proposition_lipschitz_c_transform} implies that $|f_{*,j}^c(y)-f_{*,j}^c(z)|\leq
 \operatorname{diam}(\Omega)\|y-z\|$ and $|\tilde f_{*,j}^c(y)-\tilde f_{*,j}^c(z)|\leq
 \operatorname{diam}(\Omega)\|y-z\|$.
Thus Kantorovich--Rubinstein \cite{villani2008optimal} duality implies
\begin{align*}
 \left|
 \sD_{[\tilde\mu_1,\dots,\tilde\mu_m]}(\vf_*)
 -
 \sD_{[\mu_1,\dots,\mu_m]}(\vf_*)
 \right|
 &\le
 \operatorname{diam}(\Omega)
 \sum_{j=1}^m\omega_jW_1(\tilde\mu_j,\mu_j),\\
 \left|
 \sD_{[\tilde\mu_1,\dots,\tilde\mu_m]}(\tilde\vf_*)
 -
 \sD_{[\mu_1,\dots,\mu_m]}(\tilde\vf_*)
 \right|
 &\le
 \operatorname{diam}(\Omega)
 \sum_{j=1}^m\omega_jW_1(\tilde\mu_j,\mu_j).
\end{align*}
Noting that $\sD_{[\mu_1,\dots,\mu_m]}(\tilde\vf_*)
 \le
 \sD_{[\mu_1,\dots,\mu_m]}(\vf_*)$, we know that
\begin{align}
 0
 \le
 \sD_{[\tilde\mu_1,\dots,\tilde\mu_m]}(\tilde\vf_*)
 -
 \sD_{[\tilde\mu_1,\dots,\tilde\mu_m]}(\vf_*)
 \le
 2\operatorname{diam}(\Omega)
 \sum_{j=1}^m\omega_jW_1(\tilde\mu_j,\mu_j).
 \label{eqn:bound_on_dual_difference_partial}
\end{align}
Combining \eqref{eqn:barycenter_estimation_term_0_partial}--\eqref{eqn:bound_on_dual_difference_partial} yields
\begin{align*}
 \bE W_2^2(\tilde\mu_m^{(n)},\bar\mu_m)
 &\le
 \frac{2}{\alpha^2\theta_J}
 \sum_{j\in J}\omega_j
 \bE W_2^2(\tilde\mu_j,\mu_j)
 +
 \frac{8\operatorname{diam}(\Omega)}{\alpha\theta_J}
 \sum_{j=1}^m\omega_j
 \bE W_1(\tilde\mu_j,\mu_j)\\
  &\le
 \frac{2}{\alpha^2\theta_J}
 \sum_{j\in J}\omega_j
 \bE W_2^2(\tilde\mu_j,\mu_j)
 +
 \frac{8\operatorname{diam}(\Omega)}{\alpha\theta_J}
 \Bigl( \sum_{j=1}^m\omega_j
 \bE W_2^2(\tilde\mu_j,\mu_j) \Bigr)^{1/2}.
\end{align*}
Combining with 
Lemma \ref{Lemma:jonathan_smooth_dentisy_estimation}
completes the proof.
\end{proof}

\begin{proof}[Proof of Remark \ref{remark:optimization_error}]
For each $i\in[m]$, define
\(
G_i(x,y)
:=
c(x,y)-f_{*,i}(x)-f_{*,i}^c(y)
=
\varphi_{*,i}(x)+\varphi_{*,i}^*(y)-\langle x,y\rangle.
\)
By Assumption~3,
$\varphi_{*,i}(x):=\frac12\|x\|_2^2-f_{*,i}(x)$
is $\alpha$-strongly convex on $\Omega$. 
Thus, Lemma \ref{lemma:G_f_upper_lower_quadratic-bound} implies that
\begin{equation}\label{eqn:sga_fenchel_upper_lower}
G_i(x,y)
\ge
\frac{\alpha}{2}\|x-T_{f_{*,i}^c}(y)\|^2,
\qquad
G_i(T_{f_{*,i}^c}(z),y)
\le
\frac{1}{2\alpha}\|y-z\|^2.
\end{equation}
Let $\pi_i\in\Pi(\mu_i,\tilde\mu_i)$ be an optimal coupling.
Since $T_{f_{*,i}^c}\sharp\mu_i=\bar\mu_m$,
$(T_{f_{*,i}^c},\mathrm{id})\sharp\pi_i\in\Pi(\bar\mu_m,\tilde\mu_i)$.
Using $\sum_i\omega_i f_{*,i}=0$ and
\eqref{eqn:sga_fenchel_upper_lower}, we obtain
\begin{align}\label{eqn:sga_reference_gap}
0
&\le
V_{[\tilde\mu_1,\ldots,\tilde\mu_m]}(\bar\mu_m)
- \widetilde \sD  (\vf_*) 
\leq
\sum_{i=1}^m\omega_i \int \bigl( c(T_{f_{*,i}^c}(z),y) - f^c_{*,i}(y) \bigr) d\pi_i(z,y)
\nonumber\\
&=
\sum_{i=1}^m\omega_i
\int_{\Omega\times\Omega}
G_i(T_{f_{*,i}^c}(z),y)\,d\pi_i(z,y)
\le
\frac{1}{2\alpha} \sum_{i=1}^m\omega_iW_2^2(\tilde\mu_i,\mu_i).
\end{align}
Duality therefore implies for every $t$,
\[
\widetilde \sD (\vf^{(t)}) - \widetilde \sD (\vf_*)
\le
V_{[\tilde\mu_1,\ldots,\tilde\mu_m]}(\tilde\mu_m^{(n)})-\widetilde \sD (\vf_*)
\le \frac{1}{2\alpha} \sum_{i=1}^m\omega_iW_2^2(\tilde\mu_i,\mu_i).
\]
Next, write $S_i^{(t)}:=T_{(f_i^{(t)})^c}$ and set
$Q_t:=\sum_{i=1}^m\omega_i
\int_\Omega G_i(S_i^{(t)}(y),y)\,d\tilde\mu_i(y)$.
Define \(
\varphi_i^{(t)}(x)=\tfrac12\|x\|_2^2-f_i^{(t)}(x)\) on $\Omega$ (extended by $+\infty$ outside of $\Omega$ as usual),
then for $i\in[m]$,
\begin{align*}
G_i(x_t,y)
&=\varphi_{*,i}(x_t)+\varphi_{*,i}^*(y)-\langle x_t,y\rangle
=\bigl(\varphi_{*,i}(x_t)-\varphi_i^{(t)}(x_t)\bigr)
+\bigl(\varphi_{*,i}^*(y)-(\varphi_i^{(t)})^*(y)\bigr)\\
&=(f_i^{(t)}-f_{*,i})(x_t)
+\bigl((f_i^{(t)})^c(y)-f_{*,i}^c(y)\bigr)
\end{align*}
for $ x_t=S_i^{(t)}(y)$ by Proposition \ref{property_c_transform}, and then the dual gradient formula \cite[Lemma 5]{kim2025sobolevgradientascentoptimal} and \eqref{eqn:sga_reference_gap} yield
\begin{align}\label{eqn:sga_gap_identity}
Q_t
&=
\widetilde \sD (\vf^{(t)})
-
\widetilde \sD (\vf_*)
+
\sum_{i=1}^m\omega_i
\int_\Omega
(f_i^{(t)}-f_{*,i})(S_i^{(t)}(y))
\,d\tilde\mu_i(y) \notag\\
&=
\widetilde \sD (\vf^{(t)}) - \widetilde \sD (\vf_*)
+
\langle\mathbf g_t,\vf_*-\vf^{(t)}
\rangle_{\mathcal H_\Omega} \notag\\
&\le
\frac{1}{2\alpha}\sum_{i=1}^m\omega_iW_2^2(\tilde\mu_i,\mu_i)
+
\langle\mathbf g_t,\vf_*-\vf^{(t)}
\rangle_{\mathcal H_\Omega}.
\end{align}
The SGA update satisfies
\begin{align*}
2\eta_t
\langle\mathbf g_t,\vf_*-\vf^{(t)}\rangle_{\mathcal H_\Omega}
=
\|\vf^{(t)}-\vf_*\|_{\mathcal H_\Omega}^2
-
\|\vf^{(t+1)}-\vf_*\|_{\mathcal H_\Omega}^2
+
\eta_t^2\|\mathbf g_t\|_{\mathcal H_\Omega}^2.
\end{align*}
Summing over $t$ and discarding the nonpositive terminal term
gives
\[
\sum_{t=0}^{T-1}\eta_t
\langle\mathbf g_t,\vf_*-\vf^{(t)}\rangle_{\mathcal H_\Omega}
\le
\frac12
\left(R^2+M^2\sum_{t=0}^{T-1}\eta_t^2\right).
\]
Together with \eqref{eqn:sga_gap_identity}, this implies
\begin{equation}\label{eqn:sga_averaged_gap}
\frac{1}{A_T}\sum_{t=0}^{T-1}\eta_t Q_t
\le
\frac{1}{2\alpha}\sum_{i=1}^m\omega_iW_2^2(\tilde\mu_i,\mu_i)
+
\frac{R^2+M^2\sum_{t=0}^{T-1}\eta_t^2}{2A_T}.
\end{equation}
Finally, let $\rho_i:=T_{f_{*,i}^c}\sharp\tilde\mu_i$.
The $1/\alpha$-Lipschitz property of $T_{f_{*,i}^c}$ gives
$W_2^2(\rho_i,\bar\mu_m)
\le\alpha^{-2}W_2^2(\tilde\mu_i,\mu_i)$.
Also, $(S_i^{(t)},T_{f_{*,i}^c})\sharp\tilde\mu_i$ is a coupling
of $\nu_i^{(t)}$ and $\rho_i$.
By convexity of squared Wasserstein distance under mixtures,
the triangle inequality, and \eqref{eqn:sga_fenchel_upper_lower},
\begin{align*}
W_2^2(\tilde\mu_{m,T}^{(n)},\bar\mu_m)
&\le
\frac{1}{A_T}\sum_{t=0}^{T-1}\eta_t
\sum_{i=1}^m\omega_i
W_2^2(\nu_i^{(t)},\bar\mu_m)
\le
\frac{2}{A_T}\sum_{t=0}^{T-1}\eta_t
\sum_{i=1}^m\omega_i
\bigl( W_2^2(\nu_i^{(t)},\rho_i) + W_2^2(\rho_i,\bar\mu_m) \bigr)\\
&\le
\frac{2}{A_T}\sum_{t=0}^{T-1}\eta_t
\sum_{i=1}^m\omega_i
\int_\Omega
\|S_i^{(t)}(y)-T_{f_{*,i}^c}(y)\|_2^2\,d\tilde\mu_i(y)
+\frac{2}{\alpha^2} \sum_{i=1}^m\omega_iW_2^2(\tilde\mu_i,\mu_i)\\
&\le
\frac{4}{\alpha A_T}
\sum_{t=0}^{T-1}\eta_t Q_t
+\frac{2}{\alpha^2} \sum_{i=1}^m\omega_iW_2^2(\tilde\mu_i,\mu_i)\\
&\le
\frac{4}{\alpha^2} \sum_{i=1}^m\omega_iW_2^2(\tilde\mu_i,\mu_i)
+
\frac{2}{\alpha A_T}
\left(R^2+M^2\sum_{t=0}^{T-1}\eta_t^2\right).
\end{align*}
Taking expectations and using the marginal density estimation bound Lemma \ref{Lemma:jonathan_smooth_dentisy_estimation} proves \eqref{eqn:sga_statistical_bound}. For $\eta_t=\eta_0/\sqrt{T}$, we have
$A_T=\eta_0\sqrt{T}$ and
$\sum_{t=0}^{T-1}\eta_t^2=\eta_0^2$.
For $\eta_t=\eta_0/\sqrt{t+1}$, we have
$A_T\ge\eta_0\sqrt{T}$ and
$\sum_{t=0}^{T-1}\eta_t^2\le\eta_0^2(1+\log T)$.
Substitution gives the stated iteration bounds.
\end{proof}

\section{Technical lemmas}
This section gathers some useful facts from convex analysis.
\begin{proposition}\label{property_c_transform} Let $ \eta$ be a continuous function on $\Omega$. Then at points where the relevant gradients exist, it holds true that
\begin{itemize}
\item 
$  \eta^c(x) = \frac{1}{2} \left\| T_{  \eta^c}(x) - x \right\|^2-  \eta\left(T_{  \eta^c}(x)\right) $,
\item $\eta^c(y) = \frac{1}{2} \|y\|^2 - \theta^*(y)$, with $\theta(\cdot) = \frac{\|\cdot\|^2}{2} - \eta(\cdot)$,
\item $\nabla \eta^c(y) = y - \nabla \theta^*(y)$ with $\theta(\cdot) = \frac{\|\cdot\|^2}{2} - \eta(\cdot)$,
\item For $x,y\in\Omega$,
 $\eta(x)+\eta^*(y)= \langle x,y \rangle 
  \iff x\in\argmax_{z\in \Omega}\{\langle z,y \rangle -\eta(z)\}.$
\end{itemize}
\end{proposition}

\begin{proof}
    See \cite[Chapter 1]{santambrogio2015optimal}.
\end{proof}

\begin{lemma}\label{Young_transform_curvature_upper_lower_bound}
Let $\phi:\mathbb R^d\to(-\infty,+\infty]$ be proper,
lower semicontinuous, and $\alpha$-strongly convex. Then its Fenchel conjugate
\(
\phi^*(y):=\sup_{x\in\mathbb R^d}
\{\langle x,y\rangle-\phi(x)\}
\)
belongs to $C^1(\mathbb R^d)$, with
\[
\operatorname{Lip}(\nabla\phi^*)
:=\sup_{\substack{u,v\in \bR^d\\u\neq v}}
\frac{\|\nabla\phi^*(u)-\nabla\phi^*(v)\|}{\|u-v\|}
\le\alpha^{-1}.
\]
\end{lemma}

\begin{proof}
For every $y\in\mathbb R^d$, the function
$x\mapsto\phi(x)-\langle x,y\rangle$ is proper, lower
semicontinuous, strongly convex. It therefore
has a unique minimizer $x_y$. Consequently, $\phi^*$ is finite
everywhere. The conjugate-subgradient correspondence gives
\(
\partial\phi^*(y)=\{x_y\},
\) (see \cite{rockafellar1970convex} for details of subgradients)
so $\phi^*$ is differentiable with $\nabla\phi^*(y)=x_y$.
Fix $y_1,y_2\in\mathbb R^d$ and set
$x_i:=\nabla\phi^*(y_i)$ for $i=1,2$.
By the minimizing property of $x_i$, we have
$y_i\in\partial\phi(x_i)$.
Strong convexity of $\phi$ and the Cauchy-Schwarz inequality
therefore yield
\[
\alpha\|x_1-x_2\|^2
\le
\langle y_1-y_2,x_1-x_2\rangle
\le
\|y_1-y_2\|\,\|x_1-x_2\|.
\]
It follows that
\[
\|\nabla\phi^*(y_1)-\nabla\phi^*(y_2)\|
=
\|x_1-x_2\|
\le
\alpha^{-1}\|y_1-y_2\|.
\]
Thus $\phi^*\in C^1(\mathbb R^d)$ and
$\operatorname{Lip}(\nabla\phi^*)\le\alpha^{-1}$.
\end{proof}

\begin{lemma}\label{lemma:G_f_upper_lower_quadratic-bound}
Let \(f\in C^{1,1}(\Omega)\) and suppose \(\varphi(\cdot):=\frac{\|\cdot\|^2}{2}-f(\cdot)\) is
\(\alpha\)-strongly convex on $\Omega$. Then \(f^c\) is continuously differentiable on
\(\Omega\), and
\[
 T_{f^c}(y)
 =\argmin_{x\in\Omega}\{c(x,y)-f(x)\},
 \qquad
\operatorname{Lip}(T_{f^c})
 \le\alpha^{-1}.
\]
Define
$ G_f(x,y):=c(x,y)-f(x)-f^c(y)$. Then 
\begin{align}\label{ct:gaplower}
 G_f(x,y)
 \ge\frac\alpha2\|x-T_{f^c}(y)\|^2,\qquad 
 G_f(T_{f^c}(u),v)
 \le\frac1{2\alpha} \|v-u\|^2.
\end{align}
\end{lemma}

\begin{proof}
Extend $\varphi$ by $+\infty$ outside of $\Omega$.
Since $\Omega$ is nonempty and closed and $f$ is continuous on $\Omega$, the extended-valued function $\varphi$ is proper and lower semicontinuous. Moreover, convexity of $\Omega$ and the $\alpha$-strong convexity of $\varphi$ on $\Omega$ imply that $\varphi$ is $\alpha$-strongly convex on $\mathbb R^d$ in the extended-valued sense. Thus, by
Lemma \ref{Young_transform_curvature_upper_lower_bound},
\[
\varphi^*\in C^1(\mathbb R^d),
\qquad
\operatorname{Lip}(\nabla\varphi^*)\le\alpha^{-1}.
\]
For every $y\in\Omega$, the unique maximizer
of \(x\mapsto\langle x,y\rangle-\varphi(x)\) is 
$\nabla\varphi^*(y)$ and belongs to $\Omega$ and satisfies $
\varphi(\nabla\varphi^*(y))+\varphi^*(y)=\langle \nabla\varphi^*(y),y\rangle$ by Proposition \ref{property_c_transform}.
For $x\in\Omega$ we have $c(x,y)-f(x)=\varphi(x)- \langle x, y \rangle +\frac12 \|y\|^2$. Then $ f^c(y) = \frac{1}{2} \|y\|^2 -\varphi^* (y)$, $T_{f^c}(y) =\nabla\varphi^*(y) =\argmin_{x\in\Omega}\{c(x,y)-f(x)\}$, for $ y\in\Omega$. As a consequence, $T_{f^c}(\Omega) \subset \Omega$ and $\operatorname{Lip}(T_{f^c})\le\alpha^{-1}$.
At the same time,
\begin{equation}\label{ct:FYgap}
 G_f(x,y)=\varphi(x)+\varphi^*(y)- \langle x, y \rangle.
\end{equation}
Fix $x,y\in\Omega$ and set $\bar x=T_{f^c}(y)=\nabla\varphi^*(y)$.
By Proposition \ref{property_c_transform}, $\varphi^*(y)= \langle \bar x, y \rangle -\varphi(\bar x)$. Substituting into
\eqref{ct:FYgap} and strong convexity gives
\[
 G_f(x,y)=\varphi(x)-\varphi(\bar x)-\langle y,x-\bar x \rangle
 \ge\frac\alpha2\|x-\bar x\|^2=\frac\alpha2\|x-T_{f^c}(y)\|^2.
\]
Let $u,v\in\Omega$ and set
$x=T_{f^c}(u)=\nabla\varphi^*(u)\in\Omega$. Proposition \ref{property_c_transform} implies that $\varphi(x)= \langle x, u \rangle -\varphi^*(u)$. Substituting into
\eqref{ct:FYgap}, noting that $\varphi^*$ is $\alpha^{-1}$-smooth,
\[
 G_f(T_{f^c}(u),v)
 =\varphi^*(v)-\varphi^*(u)- \langle \nabla\varphi^*(u), v-u \rangle
 \le\frac1{2\alpha}\|v-u\|^2.
 \qedhere
\]
\end{proof}

\begin{proposition}[Uniform bounds for \(c\)-transforms]\label{proposition_lipschitz_c_transform}
Let \(f:\Omega\to\mathbb R\) be any function whose \(c\)-transform
$
f^c(y):=\inf_{x\in\Omega}\left\{\frac12\|x-y\|^2-f(x)\right\}
$
is finite at some point \(y_0\in\Omega\). Then \(f^c\) is finite everywhere on \(\Omega\) and satisfies
\[
|f^c(y)-f^c(z)|
\le \operatorname{diam}(\Omega) \|y-z\|,
\qquad y,z\in\Omega.
\]
Consequently, 
\[
\|\nabla f^c\|_{L^\infty(\Omega)}
\le \operatorname{diam}(\Omega),
\qquad
\|\nabla f^c\|_{L^2(\Omega)}
\le \operatorname{diam}(\Omega) |\Omega|^{1/2}.
\]
\end{proposition}

\begin{proof}
Note that for every \(f:\Omega\to\bR\),
\[
\begin{aligned}
|f^c(y)-f^c(z)|
&\le \sup_{x\in\Omega}|c(x,y)-c(x,z)|\\
&\le
\frac12\sup_{x\in\Omega}
\bigl(\|x-y\|+\|x-z\|\bigr)\|y-z\|
\le \operatorname{diam}(\Omega) \|y-z\|.
\end{aligned}
\]
Note that $f^c(y_0)<\infty$ and $\Omega$ is bounded, $f^c$ is bounded on $\Omega$.
Since \(f^c\) is Lipschitz on $\Omega$ with 
$\|\nabla f^c\|_{L^\infty(\Omega)}
\le \operatorname{diam}(\Omega)$.
Consequently, on the cube,
\[
\|\nabla f^c\|_{L^2(\Omega)}
\le
\operatorname{diam}(\Omega)|\Omega|^{1/2}.
\]
This completes the proof. 
\end{proof}

\begin{lemma}\label{lemma:W_2H_-1_bound}
Assume that $\mu$ and $\nu$ are absolutely continuous measures
on a convex domain $\Omega$, with densities bounded from below
and from above by two (same) constants $a,b$ with
$0<a<b<+\infty$. Then we have
\[
b^{-1/2}\|\mu-\nu\|_{\dot \bH^{-1}(\Omega)}
\leq W_2(\mu,\nu)
\leq a^{-1/2}\|\mu-\nu\|_{\dot \bH^{-1}(\Omega)}.
\]
\end{lemma}

\begin{proof}
    See \cite[Theorem 5.34]{santambrogio2015optimal}.
\end{proof}

\begin{lemma}[Negative Sobolev stability of population pushforwards]
\label{lemma:population_pushforward_Hminus1_stability}

Let \(\{\mu_i,\tilde\mu_i\}_{i=1}^m\) be absolutely continuous probability measures on $\Omega$ with densities bounded between \(L\) and \(U\). Suppose that for all $j\in[m]$,
$T_j : \cK \to \Omega$ and $S_j : \Omega \to \cK$ are inverse bi-Lipschitz maps such that 
\begin{equation}\label{eq:population_derivative_bounds_lemma}
\begin{gathered}
 \alpha I\preceq DT_j\preceq \frac{U}{L\alpha^{d-1}}I\quad\text{a.e. on }\,\cK,\quad
 \frac{L\alpha^{d-1}}{U}I\preceq DS_j\preceq\alpha^{-1}I
 \quad\text{a.e. on }\,\Omega.
\end{gathered}
\end{equation}
Then, for every \(j\in[m]\), we have
\begin{equation}
\label{eq:population_pushforward_Hminus1_stability}
\frac{L^2 \alpha^{3d-2}}{U^2}
\bigl\|\tilde\mu_j- \mu_j
\bigr\|_{\dot{\bH}^{-1}(\Omega)}^2
\le
\bigl\|S_j\sharp\tilde\mu_j- S_j\sharp\mu_j
\bigr\|_{\dot{\bH}^{-1}(\cK)}^2
\le
\frac{U^d}{L^d\alpha^{d(d-1)+2}}\,
\bigl\| \tilde\mu_j - \mu_j \bigr\|_{\dot{\bH}^{-1}(\Omega)}^2.
\end{equation}
\end{lemma}

\begin{proof}

Fix $j\in[m]$ and set $\sigma:=\tilde\mu_j-\mu_j$. Then $\sigma$ has an $L^2(\Omega)$ density and $\sigma(\Omega)=0$.
For $\psi\in\dot{\bH}^1(\cK),v\in\dot{\bH}^1(\Omega)$, define mean-zero functions
$
\varphi(y)
:=
\psi(S_j(y))
-
\frac1{|\Omega|}
\int_\Omega \psi(S_j(z))\,dz
$
and 
$
w(x)
:=
v(T_j(x))
-
\frac1{|\cK|}
\int_\cK v(T_j(z))\,dz.
$
The Sobolev chain rule for bi-Lipschitz maps gives
$\varphi\in\dot{\bH}^1(\Omega)$ and $w\in\dot{\bH}^1(\cK)$.
Using \eqref{eq:population_derivative_bounds_lemma}
and change of variables, we obtain
\begin{align*}
\|\nabla\varphi\|_{L^2(\Omega)}^2
&=
\int_\Omega
\bigl|DS_j(y)^{\mathsf T}\nabla\psi(S_j(y))\bigr|^2\,dy
\le
\frac1{\alpha^2}
\int_\Omega |\nabla\psi(S_j(y))|^2\,dy\\
&=
\frac1{\alpha^2}
\int_{\cK}
|\nabla\psi(x)|^2\,|\det DT_j(x)|\,dx
\le
\frac{U^d}{L^d\alpha^{d(d-1)+2}}
\|\nabla\psi\|_{L^2(\cK)}^2,
\end{align*}
as well as
\begin{align*}
\|\nabla w\|_{L^2(\cK)}^2
&=
\int_\cK
\bigl|DT_j(y)^{\mathsf T}\nabla v(T_j(y))\bigr|^2\,dy
\le
\frac{U^2}{L^2\alpha^{2d-2}}
\int_\cK |\nabla v(T_j(y))|^2\,dy\\
&=
\frac{U^2}{L^2\alpha^{2d-2}}
\int_{\Omega}
|\nabla v(x)|^2\,|\det DS_j(x)|\,dx
\le
\frac{U^2}{L^2\alpha^{3d-2}}
\|\nabla v\|_{L^2(\Omega)}^2.
\end{align*}
Since $\sigma(\Omega)=0$, 
\begin{align*}
\left|\int_{\cK}\psi\,d(S_j\sharp\sigma)\right|
=
\left|\int_\Omega\varphi\,d\sigma\right|\le
\|\sigma\|_{\dot{\bH}^{-1}(\Omega)}
\|\nabla\varphi\|_{L^2(\Omega)}
\le
\frac{U^{d/2}}{L^{d/2}\alpha^{1+d(d-1)/2}}
\|\sigma\|_{\dot{\bH}^{-1}(\Omega)}
\|\nabla\psi\|_{L^2(\cK)}.
\end{align*}
At the same time, note that $ (S_j\sharp\sigma) (\cK) = \sigma (\Omega) = 0$ and \(
T_j\sharp(S_j\sharp\sigma)
=
(T_j\circ S_j)\sharp\sigma
=
\sigma
\).
\begin{align*}
\left|\int_{\Omega} v \,d\sigma\right|
&=
\left|\int_\cK w\,d( S_j\sharp\sigma )\right|\\
&\le
\|S_j\sharp\sigma\|_{\dot{\bH}^{-1}(\cK)}
\|\nabla w\|_{L^2(\cK)}
\le
\frac{U}{L\alpha^{(3d-2)/2}}
\|S_j\sharp\sigma\|_{\dot{\bH}^{-1}(\cK)}
\|\nabla v\|_{L^2(\Omega)}.
\end{align*}
Taking the supremum over
$\psi\in\dot{\bH}^1(\cK)$ with
$\|\nabla\psi\|_{L^2(\cK)}\le1$
(resp. $v\in\dot{\bH}^1(\Omega)$ with $\|\nabla v\|_{L^2(\Omega)}\le1$) 
yields
\begin{equation}
\label{eq:population_pushforward_Hminus1_pullback}
\frac{L^2 \alpha^{3d-2}}{U^2}
\|\sigma\|_{\dot{\bH}^{-1}(\Omega)}^2
\le
\|S_j\sharp\sigma\|_{\dot{\bH}^{-1}(\cK)}^2
\le
\frac{U^d}{L^d\alpha^{d(d-1)+2}}\,
\|\sigma\|_{\dot{\bH}^{-1}(\Omega)}^2.
\end{equation}
\end{proof}

\end{appendix}

%

\bibliographystyle{imsart-nameyear}

\end{document}